\documentclass[11pt,twoside,english,a4paper]{article}
\usepackage{babel,amssymb}
\usepackage{graphicx}
\usepackage{amsmath}
\usepackage{bm}
\usepackage[show]{ed}
\usepackage{color}

\newcommand{\x}{\boldsymbol{x}}
\newcommand{\q}{\boldsymbol{q}}
\newcommand{\p}{\boldsymbol{p}}
\newcommand{\xh}{\widehat{\x}}
\newcommand{\Sp}{\mathbb{S}^2}
\newcommand{\C}{\mathbb{C}}
\newcommand{\R}{\mathbb{R}}
\newcommand{\Z}{\mathbb{Z}}

\newcommand{\modification}[1]{{ #1}}
\newcommand{\modificationG}[1]{{ #1}}

\newcommand{\hsum}{ \sum  \hspace{-0.03in}{}^{\prime \prime}}
\newcommand{\dprime}{\hspace{-0.03in} {}^{\prime \prime}}
\newtheorem{proposition}{Proposition}[section]
\newtheorem{corollary}[proposition]{Corollary}
\newtheorem{lemma}[proposition]{Lemma}
\newtheorem{theorem}[proposition]{Theorem}
\newtheorem{hypothesis}{Hypothesis}
\newtheorem{notita}[proposition]{Remark}
\newenvironment{remark}{\begin{notita}\rm}{\hfill$\Box$\\[0.5ex]\end{notita}}

\newenvironment{proof}{{\it Proof}. }{\hfill$\Box$\\[0.5ex]}

\title{Sobolev estimates for constructive uniform-grid FFT   interpolatory approximations of spherical functions}
\author
{V. Dom\'{\i}nguez\thanks{
Departamento Ingenier\'ia Matem\'atica e Inform\'atica,  Universidad P\'ublica
de Navarra, Campus de Tudela, 31500, Tudela, Spain.
\hspace{0.01in} {\tt (victor.dominguez@unavarra.es)}}
\and
M. Ganesh\thanks{Department of Applied Mathematics \& Statistics,
Colorado School of Mines,
Golden, CO 80401. ~~~~~~~~~~~~{\tt (mganesh@mines.edu)}}
}
\date{}

\begin{document}
\maketitle 
\begin{abstract}
The fast Fourier transform (FFT) based matrix-free ansatz interpolatory  approximations 
of periodic functions are fundamental for efficient realization in several applications.
In this work we design, analyze, and implement similar constructive  interpolatory  approximations 
of  spherical functions, using samples of the unknown functions at the poles and at  the uniform  
spherical-polar grid locations $(\frac{j\pi}{N}, \frac{k \pi}{N})$,  
for $j = 1, \dots, N-1,~k=0, \dots, 2N-1$. 
The spherical matrix-free interpolation operator  range  space consists of a selective subspace of  two dimensional trigonometric polynomials 
which are rich enough to contain all spherical polynomials of  degree less than $N$.   
Using the ${\mathcal{O}}(N^2)$ data, the spherical interpolatory approximation is  
efficiently constructed by applying the FFT techniques (in both azimuthal and latitudinal variables) with only ${\mathcal{O}}(N^2 \log N)$ complexity. We describe 
the construction details using  the FFT operators and  provide complete convergence analysis of the interpolatory approximation  in the Sobolev  space framework that are well suited for 
quantification of various   computer models. 
 
We prove that the rate of spectrally accurate  convergence of the interpolatory approximations in Sobolev norms (of order zero and one)  are similar (up to a log term) to that of  the best approximation in the finite dimensional ansatz space. Efficient interpolatory quadratures on the sphere are important for  several applications including radiation transport and wave propagation computer models. We use our matrix-free  interpolatory approximations to construct robust  FFT-based quadrature rules for a wide class of non-, mildly-, and strongly-oscillatory  integrands on the sphere. We  provide numerical experiments to demonstrate  fast evaluation of the algorithm and   various theoretical results presented in the article.  
\end{abstract}

\paragraph{Key words:} Interpolation, Spherical functions, Sobolev norms, Cubature, Spherical integrals

\paragraph{AMS subject classification:} 42A15, 65D32, 33C55

\newpage 
\section{Introduction}
Approximation of  functions defined on the sphere is important for realization and understanding of various  processes described in the spherical-polar coordinate system. In particular,  approximation of an unknown spherical function with the  requirement of  exactly reproducing the function at certain
locations on the sphere, namely the interpolatory approximation, plays an important role in designing 
efficient  discrete computer models of various continuous systems.

A key tool for several large scale simulations  is the  
FFT-based   representation  (in polar coordinates) of the polynomial  interpolatory approximation of a
function defined on the circle (and in general any periodic function). The polar-coordinate/periodic case analytical
interpolatory representation facilitates fast construction of the approximation without the need to solve any 
matrix system of  interpolation constraint equations.

\modificationG{For a preferred set of   points on the sphere,  approximations of spherical functions belong to  either the non-interpolatory class 
or the interpolatory class for the set. Construction of each of these approximations can be further subdivided into either  the matrix-free class or  that require solutions of  linear algebraic systems. 
The set of points, especially for discretizing continuous systems based on differential equations (with a known source function), can be chosen for
efficiently setting up discrete computer models. For experimental data based approximations, the set of observation points are in practice 
predetermined and such data in general include noise. 

In order to avoid data sensitivity with respect to the noise, it is efficient to choose non-interpolatory class of approximations. 
There is a large literature on non-interpolatory approximations, depending on whether
the data observation points are scattered or can be chosen by the user, see for 
example~\cite{rep1, rep2, GaMhaqint:2006, rep3, rep4, rep5, rep6, rep7, rep8, rep9} and references therein.
Among these, hyperinterpolation approximations~\cite{rep2, rep4, rep7, rep8} are matrix-free and these are  global spherical polynomial
approximations with spherical harmonic Fourier coefficients (integrals on the sphere)  further approximated by a combination
of quadratures with certain degree of precision.  Quadrature-free quasi-interpolatory approximations can also be 
constructed~\cite{GaMhaqint:2006} and these are in particular suitable  for a class of scattered data.

The interpolatory class approximations have the advantage of being equal to a 
known function at all points in the set. This is in particular ideal for setting up scientific computing models governed by
differential equations with known source functions. Construction of the set  of  interpolation points and associated interpolatory
approximations is essential to develop computer models based on the collocation method, see~\cite{Cory:14, GaGrSi:1998} and references therein.

As described later in this section, our interest is on efficiently simulating partial differential equations with applications to wave propagation
and radiation transport models. Such simulations substantially benefit from the collocation method based computer models, with a fast method
to compute approximations. Matrix-free interpolatory approximations for the collocation method can be
efficiently built into the computer models, without solving any algebraic system to setup the discrete collocation system.
The FFT based evaluations of approximations  are needed for large scale simulations and precise quantification of accuracy of approximations
in Sobolev spaces is crucial in the mathematical analysis of the discrete models. The main focus of this article is on developing, analyzing,
implementing such matrix-free spherical interpolatory approximations.
}

A general approach in approximation theory is to seek a solution 
(that satisfy certain modeling constraints) in the space of polynomials.
Within the space of spherical polynomials,
it is an open problem to construct  such a powerful matrix-free 
representation of interpolatory approximations of  spherical functions. Indeed,
 if  a standard constraint that the spherical
interpolation operator (with truncated Fourier series ansatz) should  exactly reproduce  polynomials of degree, say $N\geq 3$, is imposed,  then 
it is impossible to construct a matrix-free polynomial interpolatory approximation~\cite{sloan:1}. 
This important two decade old work of Sloan~\cite{sloan:1} resulted in addressing several theoretical 
and practical questions, including efficient design of points on the sphere,  
\modificationG{see~\cite{atk_book,  sloan:3, reimer:00, sloan:2}} and extensive references therein.
It is still an open problem to prove the numerically observed $\mathcal{O}(N)$ Lebesgue constant growth  of some of the very efficient matrix-dependent spherical polynomial \modificationG{interpolation} operators~\cite{sloan:2}.

As discussed in detail in~\cite{atk_book, sloan:3, sloan:2}, the quality of spherical interpolatory approximation
(determined by the Lebesgue constant of the interpolation operator) is important. 
Further, mathematically
establishing  error estimates of the approximation is crucial for quantifying the validity of various computer models
that use such approximations. For applications, in addition to providing a fast   procedure for evaluating interpolatory
spherical approximations, it is important to prove associated error estimates in the $L^2$ and Sobolev (energy)  norms.
This is because robust error estimates in various approximate computer models  are usually established in such norms. Our main focus in this
article is on  such practical (matrix-free and FFT-based) considerations and associated robust mathematical analysis.
To this end,  we  do not require that the interpolatory spherical approximations need to  be in the space of spherical polynomials. 

	In~\cite{DoGa:2012,  GaGrSi:1998, GaMha:2006} a finite dimensional  space  $\chi_N$ (containing the space of spherical polynomials of degree less than $N$) was introduced.   
Well-posedness of the $\chi_N$-space based interpolation problem was established 
in~\cite{DoGa:2012}, using equally spaced  $2N$  azimuthal angles in $[0, 2 \pi)$ and {\em arbitrary} $N+1$ elevation latitudinal angles in $[0, \pi]$  so that the total number of interpolation points on the sphere
(including the north and south poles) is equal to the dimension of $\chi_N$. 

For the special choice of the $N+1$ non-uniform latitudinal angles  in $[0, \pi]$  that are 
based on Gauss-Lobatto points, 
as shown in~\cite{GaMha:2006}, the 
$\chi_N$-space interpolation problem is matrix-free and Sobolev error estimates for this spherical approximation  
was proved in~\cite{DoGa:2012}. If the $N+1$  latitudinal angles in $[0, \pi]$ 
are equally spaced, then the $\chi_N$-space 
uniquely solvable  interpolation problem was also shown to be matrix-free in~\cite{GaMha:2006}  
and the growth of the Lebesgue constant of this interpolation operator is only $\mathcal{O}(\log^2 N)$. This 
Lebesgue constant growth (and hence  error estimates in the uniform norm) 
was proved in~\cite{GaGrSi:1998, GaMha:2006}. As demonstrated by Sloan and
Womersley in~\cite{sloan:2} for benchmark
smooth and non-smooth functions, this non-polynomial  interpolatory spherical approximation,
with proven optimal Lebesgue constant, perform better than  several
matrix-dependent polynomial interpolatory approximations. 

Practical construction and analysis of the matrix-free interpolation operator in this article is completely different from that in~\cite{GaMha:2006}.
The main aim of this article is on the  efficient construction of uniform-grid interpolatory spherical
approximation  using only the FFT operators and to provide robust mathematical  analysis for quantifying the
interpolatory approximations in the Sobolev norm.

This article is the final of the three part constructive approximation theory 
and implementation work by Ganesh et al.~\cite{GaMha:2006, DoGa:2012} on the non-polynomial range space $\Sp$ interpolation framework introduced  in~\cite{GaGrSi:1998} (for  a 3D potential theory computer model).  Our new FFT-based construction and  \modification{Sobolev} space analysis presented in this article have potential applications in various large scale computer models  that require approximation of spherical functions.  

In particular, in our future work, we shall focus on two important classes of  specific applications of the interpolatory approximations
developed in this article:
(i) Deterministic and stochastic three dimensional wave propagation models, for evaluation of statistical quantities and 
uncertainty quantification in multiple particle configurations~\cite{GH:09, GH:11, GH:13, GH:14}; and
(ii) Advanced radiation transport (RT) computer models~\cite{transport_book, nuclear_book}.

The acoustic and electromagnetic wave propagation
Galerkin computer models developed by Ganesh and Hawkins~\cite{GH:09, GH:11, GH:13, GH:14} depend on  {\em local} spherical-polar  coordinate system based approximations (of surface currents and  integrals). These coordinate systems are  
imposed either on various patches on the surface of a single scatterer~\cite{GH:11} or  on individual particles in multiple particle configurations~\cite{GH:09,  GH:13, GH:14}. The future advanced algorithms and analysis 
for the three dimensional wave propagation models 
will be based on the  interpolatory  collocation version of deterministic and stochastic 
algorithms in~\cite{GH:09, GH:11, GH:13, GH:14}. These algorithms  will also require efficient interpolatory cubature rules for
medium to highly oscillatory integrals on the sphere.

The linear RT equation (RTE) poses a significant computational challenge, even for the next generation of super computers, 
because of the high-dimensional phase space on which it is posed. In general, the solution of the RTE
(an integro-differential equation) is a function of seven independent variables: one temporal variable, three spatial variables, one energy variable and two angular spherical-polar coordinate variables (describing the direction of radiation motion). The 
integral part of the RTE is an integral on the sphere and, in practice,  integrands with limited smoothness properties on
the sphere  (similar to functions in $\chi_N$). The industrial standard approach hitherto is to apply cubature on the
sphere with certain symmetry properties, such as that in~\cite{Cory:09}. However,
a recent derivation~\cite{Cory:14} demonstrates that interpolatory approximation based cubature  on the sphere are efficient. As discussed in the 
conclusion section in~\cite{Cory:14},
solutions to the RTE in practice are poorly behaved in the angular variables. Hence interpolatory approximations
and associated interpolatory cubature on the sphere based on approximations in the non-polynomial space $\chi_N$
will facilitate developing future advance RT computer models. 

In addition to the requirement of interpolatory  approximations  in both the above classes of applications, an important common tool required in \modification{this} future application  work is  an efficient FFT evaluation based interpolatory cubature  on the sphere with  non-, mildly-, and strongly-oscillatory integrands and  quantify the error in such cubature rules for integrands with limited smoothness properties. 
\modificationG{This article is structured as follows.}
After developing (i) an FFT-based interpolatory approximation in Section~\ref{sec:int_app}; (ii) introducing a functional 
framework in Section~\ref{sec:3}, \modificationG{based on Sobolev space decomposition~\cite{DoHeSa:2009}}; and (iii) proving the quality of our spherical interpolatory approximations in Section~\ref{sec:err_est},  we develop an efficient FFT-based interpolatory cubature on the sphere with error estimates in Section~\ref{sec:quad}. Numerical results in Section~\ref{sec:num_exp} (and Appendix~\ref{app:hyp_app}) demonstrate  various constructive and theoretical results developed and proved in this article and efficiency over a 
recent matrix-free interpolatory construction~\cite{DoGa:2012}.
We conclude this article in Appendix~\ref{sec:app_proof_sec_3} with proofs of some technical results
stated in Section~\ref{sec:3}. This will also be of independent interest for analysis/applications on 
rotationally invariant manifolds~\cite{GT:14}.

\section{An interpolatory approximation of spherical functions}\label{sec:int_app}

Let  $\Sp$ be the unit sphere in $\R^3$ parameterized, for $\xh \in \Sp$,  using the standard convention: \begin{equation}
\label{eq:x}
 \xh = \p(\theta,\phi):=(\sin\theta\cos\phi,\sin\theta\sin\phi,\cos\theta), \qquad \theta,\phi\in\R. 
\end{equation}
For any continuous \modificationG{ scalar-valued function} $F^\circ$ defined on the sphere, we  denote
$F:=F^\circ \circ \p$ and observe that 
\begin{equation}
 \label{eq:01}
 F(\theta,\phi)=F(\theta+2\pi,\phi)=F(\theta,\phi+2\pi),\quad 
 F(\theta,\phi)=F(-\theta,\phi+\pi), \quad \forall \theta,\phi\in\R. 
\end{equation}
Conversely, for any continuous \modificationG{scalar-valued function} $F$ on $\R^2$ satisfying \eqref{eq:01}, there exists a
unique associated function $F^\circ$ on the sphere. Motivated by this observation, we define the space
\begin{equation}
 {\cal C}:=\left \{F:\mathbb{R}^2\to\mathbb{C}\ :\ F \text{ is continuous and satisfies \eqref{eq:01}} \right \}
 \label{eq:02}
\end{equation}
which in view of \eqref{eq:01} can be identified with ${\cal C}(\Sp)$, the space of
complex valued continuous  \modificationG{scalar-valued functions}  on the unit sphere. 

In this work we will introduce  a trigonometric interpolant for functions $F^{\circ}$ on ${\cal C}(\Sp)$ 
using the following details: For $N\in\mathbb{N}$ with $N \geq 2$,  consider the equally spaced grid points 
\[
 \theta_j=\frac{j\pi}{N},\quad \phi_k =\frac{k\pi}{N},\qquad \quad j,k\in\mathbb{Z}. 
\]
We recall that, for {\em any} $\phi \in \mathbb{R}$,  the north and south poles are respectively 
$\p(\theta_0, \phi) = {\bf n}$ and $\p(\theta_N, \phi) = {\bf s}$.
Using  \eqref{eq:x},  the parametrized uniform grid  
\begin{equation}\label{eq:uni_grid}
\mathcal{G}_N = \left\{(\theta_j,\phi_k): j = 0, \dots, N,  k = 0, \dots, 2N-1 \right\}
\end{equation}
in $[0, \pi] \times [0, 2 \pi)$  corresponds to a grid of $2N(N-1)+2$ distinct points on the sphere. We also note that,
similar to \eqref{eq:01}, we have 
\begin{equation}\label{eq:symmetry}
 \p(\theta_j,\phi_k)=\modification{
 \p(\theta_{j+2N},\phi_{k})=
 \p(\theta_j,\phi_{k+2N})},\quad 
 \p(\theta_j,\phi_k)= \p(\theta_{j+N},\phi_{k-N}). 
\end{equation}
It is convenient to introduce the space of even and odd functions
\begin{eqnarray}
 \mathbb{D}_N^{\rm e}&:=&{\rm span}\,\langle \cos j\theta\ :\ j=0,\ldots,
N\rangle
 ={\rm span}\,\langle \cos^j\theta\ :\ j=0,\ldots,
N\rangle\\
 \mathbb{D}_N^{\rm o}&:=&{\rm span}\,\langle \sin j\theta\ :\ j=1,\ldots,
N+1\rangle={\rm span}\,\langle \sin\theta \cos^j\theta\ :\ j=0,\ldots,
N\rangle,
\end{eqnarray}
and then, as in~\cite{DoGa:2012,  GaGrSi:1998, GaMha:2006},  we consider a $2N^2-2N+2$ 
dimensional subspace of ${\cal C}$, defined as
\begin{eqnarray}
 \chi_N&:=&
\bigg\{p_0(\theta)+\!\!\sum_{ -N <m \le N \atop  \text{even }m\ne0}\!\!\!\!
\sin^2\theta\, p_m(\theta)\exp({\rm i}m\phi) +\!\!
\sum_{ -N <m \le N\atop   \text{odd } m}\!\!
  p_m(\theta)\exp({\rm i}m\phi)\ : \nonumber\\
&&
 \ \hspace{3cm}\ p_0\in {\mathbb D}^{\rm e}_N,\ 
  p_{2\ell}\in
{\mathbb D}^{\rm e}_{N-2},\ \ell \ne 0, \ p_{2\ell+1}\in
{\mathbb D}^{\rm o}_{N-2}\bigg\}\label{eq:chin}\\
&=& \bigg\{\sum_{ -N  <m \le N\atop  \text{even }m}\!\!\!\!
  p_m(\theta)\exp({\rm i}m\phi) +\!\!
\sum_{ -N <m \le N\atop   \text{odd } m}\!\!
  p_m(\theta)\exp({\rm i}m\phi)\  : \nonumber\\
&&
 \ \hspace{3cm}\ p_{2\ell}\in {\mathbb D}^{\rm e}_N,\ 
  p_{2\ell}(0)=p_{2\ell}(\pi)=0 \text{ for $\ell \ne 0$}, \ p_{2\ell+1}\in
{\mathbb D}^{\rm o}_{N-2}\bigg\}.\ \quad\label{eq:chin2}
\end{eqnarray} 
The equality between \eqref{eq:chin} and \eqref{eq:chin2}  follows from the fact that  for $p\in\mathbb{D}_N^{\rm e}$
\[
p(0)=p(\pi)=0\quad \Longleftrightarrow\quad p (\theta)=\sin^2\theta\
q(\theta),\qquad\text{with } q\in\mathbb{D}^{\rm e}_{N-2}.
\]
We refer to \cite[Section~2]{GaGrSi:1998}  for details of arriving at the subspace 
$\chi_N$ of ${\cal C}$ from the standard trigonometric polynomial two dimensional Fourier  approximation space on $[0, 2\pi] \times [0, 2\pi]$.

It is easy to check that any function $F_N\in \chi_N$ satisfies
\eqref{eq:01}. In other words, 
the elements of $\chi_N$  can be identified with  continuous functions on
the sphere.  Next we consider an interpolation
problem with $2N^2-2N+2$ interpolatory points on the unit sphere. 

The spherical $\chi_N$-interpolatory approximation  problem  is defined as follows: For any $F\in{\cal C}$, 
\begin{equation} \label{eq:interpolant}
\text{find}~ {\cal Q}_NF \in \chi_N,  \,\quad \text{such that}\quad {\cal
Q}_NF(\theta_j,\phi_k)=F(\theta_j,\phi_k),\quad j=0,\ldots,N,\quad k=0,\ldots,
2N-1. 
\end{equation}
In \cite[Proposition~1]{DoGa:2012} we  proved that the interpolation problem on $\chi_N$, 
with  the $2N$ uniform grid  azimuthal angles $\phi_k$ and {\em arbitrary} $N+1$ 
latitudinal points $\theta \in [0, \pi]$,  is uniquely solvable and hence ${\cal Q}_N$ reproduces
functions in $\chi_N$. 
 In fact the unique solution to the
interpolation problem can be expressed analytically without the need to solve
any linear system. That is, \eqref{eq:interpolant}  is a matrix-free interpolation problem. 

In this article, we present  a constructive (FFT operators based) matrix-free
proof of this result, adapted to the  particular choice of   uniform grid latitudinal points, for two reasons: (a) it shows, from a practical point of view,
how the interpolant can be fast computed employing only the FFT techniques; (b) it will provide
an indication on how the analysis of the convergence of the interpolatory approximation
${\cal Q}_NF$ to $F$ could be carried out. The FFT friendly uniform grid latitudinal points provide a challenging 
Sobolev space analysis framework on the sphere compared to that with latitudinal points that are zeros of certain orthogonal polynomials~\cite{DoGa:2012}. Our analysis in this article 
leads to an interesting  trigonometric polynomials based one dimensional inequality  conjecture 
(that we could numerically verify for  practically useful cases). 

Next we consider some FFT based operators that we use in the construction of the matrix-free
interpolatory approximation. Let  ${\bf DST}_N:\mathbb{C}^{N-1}\to \mathbb{C}^{N-1}$ and 
${\bf DCT}_{N}:\mathbb{C}^{N+1}\to \mathbb{C}^{N+1}$  
 denote the  discrete sine and cosine  transform ({\em of type I}) defined as 
\begin{align}
\left({\bf DST}_N({\bf y})\right)_j&:=    
\sum_{k=1}^{N-1} y_k \sin\left(\frac{kj\pi}{N}\right) ,&& j=1,\ldots,N-1,  \label{eq:dst_op}\\
\left({\bf DCT}_N({\bf x})\right)_j&:=
\sum_{k=0}^{N}\dprime  x_k \cos\left(\frac{kj\pi}{N}\right) ,&& j=0,\ldots,N,  \label{eq:dct_op}
\end{align}
where $\hsum$ means the usual summation with only half the first and last terms included.
These operators can easily be constructed using the standard FFT operator.
Let ${\bf iDST}_N$ and ${\bf iDCT}_N$ respectively denote the inverse discrete sine and cosine transforms. We will also need  the inverse of the  discrete Fourier transform ${\bf iFFT}_{M}:\mathbb{C}^{M}\to
\mathbb{C}^{M}$,
\begin{eqnarray}
\left({\bf iFFT}_M({\bf z})\right)_j&:=&\frac{1}{M} \sum_{k=0}^{M-1} z_k \exp\Big(-\frac{2kj\pi\;\mathrm{i}}{M}\Big),\quad j=0,\ldots,M-1, \label{eq:ifft_op}
\end{eqnarray} 
with the associated forward  discrete Fourier transform denoted by ${\bf FFT}_{M}$.
For construction of the matrix-free interpolation operator we will use only the above explicitly defined operators. Their inverses are defined in the following proposition and are used to prove the properties of  the 
spherical interpolatory operator ${\cal Q}_N$  in \eqref{eq:interpolant}.

\begin{proposition}\label{prop:prop01}
For any  given data 
\[
 F_{j,k}=F(\theta_j,\phi_k),\quad j=0,\ldots,N,\quad k=0,\ldots,
2N-1, 
\]
representing a function $F \in {\cal C}$ at the grid locations $\mathcal{G}_N$ in~\eqref{eq:uni_grid}, 
the well defined interpolatory approximation in \eqref{eq:interpolant} can
be efficiently constructed using the  FFT appropriate matrix-free ansatz
\begin{equation}
\begin{split}
 ({\cal Q}_NF)(\theta,\phi)&=\frac{2}{N}   \sum_{0\le m\le  N/2} \hspace{-0.2in} 
 \hspace{0.2in}   \bigg[\sum_{\ell=0}^{N} \dprime \alpha_\ell^{2m}\cos \ell \theta \bigg]\exp(2m\mathrm{i}\phi)\\
 &\quad +
  \frac{2}{N} \hspace{-0.1in} \sum_{-N/2<m\le -1} \hspace{-0.2in} 
 \hspace{0.2in}   \bigg[\sum_{\ell=0}^{N} \dprime \alpha_\ell^{2m+2N}\cos \ell \theta \bigg]\exp(2m\mathrm{i}\phi)
  \\
&\quad +
 \frac{2}{N} \hspace{-0.05in}  \sum_{1 \leq  n \leq  (N+1)/2}  
 \bigg[\sum_{\ell=1}^{N-1} \beta_\ell^{2n-1}\sin \ell \theta\bigg]\exp((2n-1)\mathrm{i}\phi), \\
&\quad +
 \frac{2}{N} \hspace{-0.05in}  \sum_{(-N+1)/2< n \leq  0}  
 \bigg[\sum_{\ell=1}^{N-1} \beta_\ell^{2n-1+2N}\sin \ell \theta\bigg]\exp((2n-1)\mathrm{i}\phi),
  \end{split}\label{eq:01:Prop01}
  \end{equation}
 where the coefficients $\left(\alpha^{2m}_\ell \right)_{\ell=0}^N $ and 
  $\left(\beta^{2n-1}_\ell \right)_{\ell=1}^{N-1} $, for $m=0,\ldots,N-1$ and $n=1,\ldots,N$,  can be computed using the data and the following fast algorithm:
\begin{enumerate}
\item Compute inverse transform data
\begin{equation}\label{eq:ifft}
 (f_{j,m})_{m=0}^{2N-1}:={\bf iFFT}_{2N}((F_{j,k})_{k=0}^{2N-1} ),\quad j=0,\ldots,N.
\end{equation}
\item 
Then compute the coefficients in~\eqref{eq:01:Prop01} using the sine and cosine
transforms as 
\begin{subequations}\label{eq:dsc-dst}
\begin{align}
\left(\alpha^{2m}_\ell \right)_{\ell=0}^N&:=  {\bf DCT}_N((f_{j,2m})_{j=0}^{N}),&& m=0,\ldots, N-1,\label{eq:dct}\\
\left(\beta^{2n-1}_\ell\right)_{\ell=1}^{N-1}&:= {\bf DST}_N((f_{j,2n-1})_{j=1}^{N-1}),&& n=1,\ldots, N.\label{eq:dst}
\end{align}
\end{subequations}
\end{enumerate}

\end{proposition}
\begin{proof} With  coefficient vectors ${\boldsymbol \alpha}$ and ${\boldsymbol \beta}$ as in 
\eqref{eq:01:Prop01}-\eqref{eq:dsc-dst}, we first define,
for $0\le m\le  N/2$ and   $1\le  n\le  (N+1)/2$, the
 even and odd functions: 
\[
 p_{2m} (\theta)= \frac{2}{N}\sum_{\ell=0}^{N}\dprime \alpha_\ell^{2m}\cos \ell \theta \in \mathbb{D}_N^{\rm e},\qquad \qquad 
 p_{2n-1}(\theta)= \frac{2}{N}\sum_{\ell=1}^{N-1} \beta_\ell^{2n-1}\sin \ell\theta \in   \mathbb{D}_{N-2}^{\rm o}, 
\]
and similarly define for $-N/2< m\le  -1$ and  $(-N+1)/2<n\le 0$ the even and odd functions:
\[
 p_{2m} (\theta)= \frac{2}{N}\sum_{\ell=0}^{N}\dprime \alpha_\ell^{2m+2N}\cos \ell \theta \in \mathbb{D}_N^{\rm e},\qquad \qquad 
p_{2n-1}(\theta)= \frac{2}{N}\sum_{\ell=1}^{N-1} \beta_\ell^{2n-1+2N}\sin \ell\theta \in   \mathbb{D}_{N-2}^{\rm o}.
\]
In particular, for $0\le m\le  N/2$,  $1\le  n\le  (N+1)/2$, 
$j = 0, \dots, N$, and  $k = 1, \dots, N-1$, we have 
\[
 p_{2m} (\theta_j)= \Big({\bf iDCT}_N(\alpha_{\ell}^{2m})_{\ell=0}^{N}) \Big)_j, 
 \qquad \qquad 
 p_{2n-1}(\theta_k)= \Big({\bf iDST}_N(\beta_{\ell}^{2n-1})_{\ell=1}^{N-1}) \Big)_k.
 \]
Using \eqref{eq:chin2} and \eqref{eq:01:Prop01},  to prove that ${\cal Q}_NF\in\chi_N$, it is sufficient 
to show  that $p_{m}(0)=p_m(\pi)=0$ for  $m\ne 0$.
Applying property \eqref{eq:01}, we obtain 
\[
 F_{0,k} =F(0,\phi_k)=F(0,\cdot) ,\quad  F_{N,k} =F(\pi,\phi_k)=F(\pi,\cdot),\qquad\text{for }
 k=0,\ldots,2N-1, 
\]
and therefore, $(f_{0,m})_{m=0}^{2N-1}$ and $(f_{N,m})_{m=0 }^{2N-1}$ are the result of applying 
 the iFFT operator to constant vectors. Thus,
\begin{equation} \label{eq:01:ProofProp1}
 f_{0,m} =\begin{cases}
                   F(0,0), &\text{if $m=0$},\\
                   0,& \text{otherwise},
                  \end{cases}\qquad
 f_{N,m} =\begin{cases}
                   F(\pi,0), &\text{if $m=0$},\\
                   0,& \text{otherwise.}
                  \end{cases} \quad
 \end{equation}
 Furthermore taking into account \eqref{eq:01:ProofProp1}, we easily 
see that for $0\le m\le N/2$ it holds
\begin{eqnarray}\label{eq:02:ProofProp1}
 p_{2m}(\pi)&=&\frac{2}{N}\sum_{\ell=0}^{N}\dprime \alpha_\ell^{2m}\cos(\ell\pi) =   
 \Big({\bf iDCT}_N((\alpha_\ell^{2m})_{\ell=0}^N) \Big)_N\nonumber\\
 &=& \Big({\bf iDCT}_N({\bf DCT}_N((f_{j,2m})_{j=0}^{N})) \Big)_N
 =f_{N,2m}=\begin{cases}
                   F(\pi,0), &\text{if $m=0$},\\
                   0,& \text{otherwise.}
                  \end{cases}
\end{eqnarray}
For $-N/2\le m\le -1$, proceeding analogously we obtain 
\[
  p_{2m}(\pi)=\frac{2}{N}\sum_{\ell=0}^{N}\dprime \alpha_\ell^{2m+2N}\cos(\ell\pi) =  
  \Big({\bf iDCT}_N({\bf DCT}_N((f_{j,2m+2N})_{j=0}^{N})) \Big)_N
 =f_{N,2m+2N}=0.
\]

Similarly we derive 
\begin{equation}\label{eq:03:ProofProp1}
 p_{2m}(0)=f_{0, 2m}=\begin{cases}
                   F(0,0), &\text{if $m=0$},\\
                   0,& \text{otherwise}.
                  \end{cases}
 \end{equation}
Thus  ${\cal Q}_N F\in\chi_N$ and that
\begin{equation}\label{eq:04:ProofProp1}
  {\cal Q}_N F(\theta_j,\phi_k)=F(\theta_j,\phi_k),\quad  j\in\{0,N\},\quad k=0, \ldots, 
2N-1. 
 \end{equation} 
 To check that ${\cal Q}_N F$  interpolates $F$  at the rest of the grid points in $\mathcal{G}_N$,  we 
can use a similar argument. For $j=1,\ldots, N-1$ and $k=0,\ldots,2N-1$, using 
\[
\begin{split}
({\cal Q}_NF)(\theta_j,\phi_k)&=  \hspace{-0.1in} \sum_{-N/2<m\le  N/2} \hspace{-0.2in} 
 \hspace{0.2in} p_{2m} (\theta_j)  \exp(2m\mathrm{i}\phi_k) + 
   \hspace{-0.14in}  \sum_{(-N+1)/2< n \leq  (N+1)/2}  
 p_{2n-1}(\theta_j)
\exp((2n-1)\mathrm{i}\phi_k) \\
  &=\sum_{0\le m\le  N/2}   \Big({\bf iDCT}_N((\alpha_\ell^{2m})_{\ell=0}^N) \Big)_j \exp(2m\mathrm{i}\phi_k)\\
  &\quad + 
   \sum_{1\le n \leq  (N+1)/2} 
  \Big({\bf iDST}_N((\beta_\ell^{2n-1})_{\ell=1}^{N-1}) \Big)_j \exp((2n-1)\mathrm{i}\phi_k) \\
   &\quad +\sum_{-N/2<m\le  -1}   \Big({\bf iDCT}_N((\alpha_\ell^{2m+2N})_{\ell=0}^N) \Big)_j \exp(2m\mathrm{i}\phi_k) \\
  &\quad+ 
   \sum_{(-N+1)/2< n \leq 0} 
  \Big({\bf iDST}_N((\beta_\ell^{2n-1+2N})_{\ell=1}^{N-1}) \Big)_j \exp((2n-1)\mathrm{i}\phi_k)
  \\
  &=\sum_{n=0}^N   f_{j,n}\exp\Big(\frac{n\mathrm{i}\pi k}{N}\Big)+
   \sum_{n=-N+1}^{-1}   f_{j,n+2N}\exp\Big(\frac{n\mathrm{i}\pi k}{N}\Big)
\\
  &=\sum_{n=0}^N   f_{j,n}\exp\Big(\frac{n\mathrm{i}\pi k}{N}\Big)+
   \sum_{n=N+1}^{2N-1}   f_{j,n}\exp\Big(\frac{(n-2N)\mathrm{i}\pi k}{N}\Big)
  \\
 &=\sum_{n=0}^{2N-1}   f_{j,n}\exp\Big(\frac{n\mathrm{i}\pi k}{N}\Big)=\left({\bf FFT}_{2N}((f_{j,n})_{n=0}^{2N-1})\right)_k\\
 &=
  \left({\bf FFT}_{2N}\left({\bf iFFT}_{2N}\big((F_{j,k})_{k=0}^{2N-1}\big)\right)\right)_k
  =F_{j,k}=F(\theta_j,\phi_k), 
\end{split}
\]
where in the penultimate step  we have used 
$\exp(\frac{(n-2N)\mathrm{i}\pi k}{N})= \exp (\frac{n\mathrm{i}\pi k}{N})$.
Combining this result with \eqref{eq:04:ProofProp1}, we proved that the matrix-free representation 
in~\eqref{eq:01:Prop01} solves the interpolation problem~\eqref{eq:interpolant}.
The uniqueness of the interpolant follows either from similar arguments or 
as a consequence of the existence of the solution since the underlying matrix in 
the interpolation problem is square.
\end{proof}

\begin{remark}
It is easy to see that the matrix-free representation \eqref{eq:01:Prop01} also provides
fast evaluation in the azimuthal variable $\phi$, using the FFT. 
We note  that the process described in Proposition \ref{prop:prop01} does not 
need that $F\in{\cal C}$: It suffices $F$ to be a continuous function in $[0,\pi]\times [0,2\pi]$, but if $F\notin{\cal C}$, the interpolant is  a trigonometric polynomial which need not be in $\chi_N$.  
\end{remark}

\begin{remark}
Roughly speaking the process explained in Proposition \ref{prop:prop01} consists 
in applying the FFT to the matrix $F_{kj}$ by columns and then the 
DCT and DST to the even and odd rows respectively. Obviously, we can also revert the order of application 
of these transformations. In any case, the calculations are fast requiring only ${\cal O}(N^2\log N)$ operations
and are parallelizable. Moreover, the matrix-free representation can be exploited for performing 
fast evaluations of the interpolatory approximation for example, on dyadic grids, or by combining with 
appropriate piecewise polynomial interpolation. 
\end{remark} 

We conclude this section by presenting the main result of this article, namely 
the convergence of the matrix-free interpolatory approximation in Sobolev norms $\|\cdot\|_{{\cal H}^s(\Sp)}$ 
on the sphere for $s\in[0,1]$ with
the regularity of spherical functions to be approximated also  measured in Sobolev norms. We introduced these spaces, in terms 
of spherical harmonics,  in the next section. Before presenting the result, 
we need the following technical hypothesis. 
Although we do not have a proof of the hypothesis, at least for all practical cases, we have numerically verified   that  the hypothesis is true: In Appendix~\ref{hypo:01} we demonstrate that   the hypothesis is true for any integer 
$2 \leq N\le 2^{14}$. The $N = 2^{14}= 16,384$ case correspond to 
the spherical interpolation problem with over $500$ million data locations. Thus
we have verified the hypothesis for almost all practical application of the interpolant
studied in this article. In Appendix~\ref{hypo:01}, we provide details of how we numerically verified the hypothesis.

\begin{hypothesis}\label{hypo:01} \modificationG{({\bf Numerically verified in  Appendix~\ref{hypo:01}}):} \\
For each $j = 1, 2, 3$, there exists $c_{\rm H}^{(j)} <1$, independent of $N$, so that
\begin{equation}\label{eq:H:00}
 -2\int_{0}^\pi |p_N(\theta)|^2\cos(2N\theta)\sin \theta \,{\rm d}\theta\le 
 c_{\rm H}^{(j)}\int_{0}^\pi |p_N(\theta)|^2\sin \theta \,{\rm d}\theta,\qquad
\forall
p_N\in A_N^j, 
\end{equation}
where 
\begin{equation}\label{eq:H:01}
A_N^1:=\mathbb{D}_{N-2}^{\rm e}, \,\quad  A_N^2:=
\mathbb{D}_{N-2}^{\rm
o},\quad  \text{and} \quad A_N^3:=\{\sin^2(\theta)q_{N-2}(\theta) :\:
q_{N-2}\in
\mathbb{D}_{N-2}^{\rm e}\}\subset\mathbb{D}_N^{\rm e}.
\end{equation}
\end{hypothesis}
 
Now we state the main theoretical spectrally accurate convergence result of the article.
\begin{theorem}\label{theo:main} Suppose that Hypothesis \ref{hypo:01} holds. Then, for
$F \in {{\cal H}^\modification{t}}$ with \modification{$\modification{t}>5/2$}  there exists  $C_{\modification{t}}>0$ so that, for  $s\in[0,1]$,
\begin{equation}\label{eq:main_res}
 \|F-{\cal Q}_N F\|_{{\cal H}^s}\le C_r N^{s-\modification{t}}(\log N)^{s/2}\|F\|_{{\cal H}^\modification{t}}. 
\end{equation}
\end{theorem}

\begin{remark}
A similar estimate in the continuous function space norm ($\| \cdot \|_\infty$) for the spherical interpolation operator
(using Chebyshev polynomial basis based matrix-free representation) was proved in~\cite{GaMha:2006}: 
\[
 \|{\cal F}-{\cal Q}_NF\|_\infty \le C (\log N )^2 N^{-m}\|F\circ
{\bf x}^{-1}\|_{{\cal C}^m(\Sp)},
\]
where ${\cal C}^m(\Sp)$ denotes  the space of functions on $\Sp$ with 
continuous derivatives  up to order $m$,  endowed with the natural norm.
We recall that ${\cal H}^{1+\epsilon}(\Sp) \subset {\cal C}^0(\Sp)$, for any $\epsilon > 0$.

Convergence analysis of Galerkin computer models  of partial differential equations (PDEs)
are usually studied in the Hilbert space setting Sobolev (and equivalent  energy) norms and hence our new
result is widely applicable, for example, in  analyzing fully discrete Galerkin methods for  approximating PDE
(and its equivalent boundary integral equation) based models. Fully discrete Galerkin methods are obtained
by approximating Galerkin integrals (and also  integral operators) in the model by
finite sums (quadratures/cubatures). In  Section~\ref{sec:quad} we demonstrate
the applicability of the Sobolev norm estimate~\eqref{eq:main_res} for analyzing efficient interpolatory cubatures.
\end{remark}

\begin{remark}
In~\cite{DoGa:2012} a similar interpolation process was studied  with a non-uniform 
distribution of the nodes in the elevation angle, namely, that which
makes  $\cos \theta_n$ the Gauss-Lobatto points. The 
convergence in this non-uniform grid case was shown to be very similar to that stated in Theorem~\ref{theo:main} but without the penalizing 
$\log N$ term.  As demonstrated in Section~\ref{sec:num_exp}, the FFT-based approximation considered in
this article is computationally more efficient than that in~\cite{DoGa:2012}. However, the mathematical analysis for the equally spaced
grid points  case is challenging  in the Sobolev  framework as shown 
in the next two sections.

\end{remark}

\section{Functional framework and properties}\label{sec:3}

In this section we describe orthogonal decompositions of some Sobolev spaces \modificationG{\cite{DoHeSa:2009}}  that are crucial for proving Theorem 
\ref{theo:main}. To this end, we introduce some fundamental properties of various
norms that we state in this section and prove these properties in Appendix~\ref{sec:app_proof_sec_3}. 

\subsection{Spherical harmonics and Sobolev spaces on the sphere}

The Sobolev spaces on the unit sphere can be introduced in several equivalent forms. 
One can work, for instance, with an atlas of the surface,  associated local charts and partition of unity functions and defined them in terms of $H^{s}(\mathbb{R}^2)$. This  general approach is valid for any sufficiently smooth 
surface~\cite{ad:2003, McLean:2000}. We may also construct the Sobolev spaces on the sphere
as a Hilbert scale using the eigenfunctions of the Laplace-Beltrami operator, namely, the spherical harmonics. 
We follow the spectral approach~\cite{Ne:2001} for functional framework and introduce essential details that
we use throughout this article. 

Using the associated Legendre polynomial
\begin{equation}
\label{eq:Legendres}
P_n^m(x):=
\frac{(-1)^m}{2^n n!} (1-x^2)^{m/2}\frac{{\rm d}^{m+n}}{{\rm
d}x^{m+n}}(x^2-1)^n, 
\end{equation}
we define 
\begin{equation}\label{eq:qLeg}
Q_n^m(\theta):=\bigg(\frac{2n+1}2 \frac{(n-m)!}{(n+m)!}
\bigg)^{1/2}P_n^{|m|}(\cos \theta),\quad
Q_{n}^{-m}:=Q_{n}^m, \qquad 0\le m\le n,\quad n=0,1,\ldots
\end{equation}
Denoting
\begin{equation}\label{eq:em}
e_m(\phi) :=\frac{1}{\sqrt{2\pi}}\exp({\rm i}m\phi),\qquad m\in\Z, 
\end{equation}
we  introduce the spherical harmonics~\cite{atk_book,Ne:2001}, a  polynomial
of degree $n$  on the $\Sp$, as
\begin{equation}
\label{eq:defHarmonicSpherics}
Y_n^m(\theta,\phi):= (-1)^{(m+|m|)/2} Q_{n}^m( \theta) e_{m}(\phi),\qquad
m=- n ,\ldots,n,\quad n=0,1,\ldots.
\end{equation}
It is well known that $\{Y_n^m \ : \  n = 0, 1, 2, \ldots, |m| \leq n\}$ 
is an orthonormal basis of 
\[
 {\cal H}^0:=\Big\{F:\R^2\to \R\  :\  F \text{ satisfies \eqref{eq:01}},\:
\int_0^\pi\!\int_0^{2\pi} |F(\theta,\phi)|^2\,\sin\theta\,{\rm
d}\phi\,{\rm d}\theta<\infty \Big\},
\]
endowed with the natural inner product and the induced norm $ \|F\|_{{\cal H}^0}$.
That is, if we define for any $F\in {\cal H}^0$, 
\begin{equation}
 \label{eq:defFnm}
 \widehat{F}_{n,m}:=\int_0^\pi\int_0^{2\pi}F (\theta,\phi)
\overline{{ Y_{n}^m(\theta,\phi)}}\sin\theta\,{\rm d}\phi \,{\rm d}\theta
\end{equation}
then
\[
 F=\sum_{m=-\infty}^\infty \sum_{n=|m|}^\infty \widehat{F}_{n,m} Y_n^m, \quad 
 \|F\|_{{\cal H}^0}^2=\sum_{m=-\infty}^\infty \sum_{n=|m|}^\infty |F_{n,m}|^2.
\]
We recall that for any $F^{\circ}:\Sp\to\mathbb{C}$, we have denoted  $F=F^\circ\circ{\bf p}$.
We follow the standard convention to identify $[Y_{n}^m]^{\circ}$ with $Y_{n}^m$, using
the identity $[Y_{n}^m]^{\circ} = Y_{n}^m \circ{\bf p}$. 
Clearly, if ${\cal L}^2(\Sp)$ denotes the space of all square integrable functions on $\Sp$, we have 
\[
 {\cal H}^0 =\{F : \ F^\circ \in{\cal L}^2  (\Sp)\},\quad \text{with}\quad
 \|F\|_{{\cal H}^0}=\|\modification{F^\circ}\|_{{\cal L}^2(\Sp)}. 
\]
The Sobolev spaces ${\cal H}^s$ for $s\in\mathbb{R}$, and their counterparts 
${\cal H}^s(\Sp)$, can be defined proceeding
analogously. Hence, the Sobolev norm of order $s$ is given by
\[
 \|F\|^2_{{\cal H}^s}:=\sum_{m=-\infty}^\infty \sum_{n=|m|}^\infty \left(n+\tfrac12\right)^s 
|\widehat{F}_{n,m}|^2, 
\]
which is well defined for instance if  $F\in \mathbb{T}:={\rm span}\: \langle
Y_n^m: n = 0, 1, \dots, |m| \leq n \rangle$. We may also define ${\cal H}^s$  as the completion of
$\mathbb{T}$ in $ \|\:\cdot\:\|_{{\cal H}^s}$. Finally, the Sobolev space on $\Sp$ can be  defined as 
\[
{\cal H}^s(\Sp):=\{F^\circ \ : \ F\in{\cal H}^s\}.  
\]
\subsection{Sobolev-like spaces for the Fourier modes and properties}

\modification{
We will introduce now an orthogonal decomposition of the Sobolev spaces ${\cal H}^s$ which will play an essential role in the analysis of the convergence of our interpolator.  This decomposition, first introduced in \cite{DoHeSa:2009}, consists essentially in  periodic one variable functions in $\theta$ which are Fourier coefficients, in $\phi$, of functions on the sphere.}

Given $f\in L_{\rm loc}^1(\R)$ we denote
\[
 (f\otimes e_m)(\theta,\phi):=f(\theta)e_m(\phi), \qquad m\in\Z.
\]
For $s\ge 0$, we can define the spaces
\[
 W_m^s:=\{f\in L_{\rm loc}^1 (\R) :\ f\otimes e_m\in {\cal H}^s\}, 
\]
endowed with the image norm
\[
 \|f\|_{W_m^s}:=\|f\otimes e_m\|_{{\cal H}^s}.  
\]
Then
\begin{equation}
\label{eq:normWmr}  f = \sum_{n=|m|}^\infty
 \widehat{f}_m(n) Q_n^m,\qquad
 \|f\|_{W_m^s}=\bigg(\sum_{n=|m|}^\infty
(n+{\textstyle\frac12})^{2s}|\widehat{f}_m(n)|^2\bigg)^{1/2}
\end{equation}
with convergence in $W_m^s$, where
\begin{equation*}
 \widehat{f}_m(n):=\int_{0}^\pi f(\theta)Q_n^{m}(\theta)\sin\theta\,{\rm
d}\theta=\int_0^\pi \int_0^{2\pi} \big(f\otimes e_m \big)(\theta,\phi)
\overline{Y_n^m(\theta,\phi)}\, \sin\theta\, {\rm
d}\phi\,{\rm d} \theta = \widehat{(f\otimes e_m)}_{n,m}. 
\end{equation*}
Clearly, $W_m^s=W_{-m}^s$ and for  $r> s$ the injection $W_m^r\subset W_m^s$ is
compact. Moreover, using  \eqref{eq:normWmr},
\begin{equation}
 \label{eq:inequatityWms}
\|f\|_{W_m^s}\le   (|m|+{\textstyle\frac12})^{s-r} \|f\|_{W_m^r},\qquad \forall
r\ge s.
\end{equation}

We note that \eqref{eq:01} imposes  periodicity and parity conditions on the
elements of $W_m^s$, namely
\begin{equation}
\label{eq:periodicity-conditions}
f \in W_m^s  \quad \Longrightarrow\quad f(\,\cdot\,+ 2\pi)=f,\quad
f(-\,\cdot\,)=(-1)^m f. 
\end{equation}
If we define the mapping 
\[
 ({\cal F}_m F\big)(\theta) := \sum_{n=|m|}^\infty \widehat{F}_{n,m} Q_n^m
(\theta) =\int_0^{2\pi} F(\theta,\phi)e_{-m}(\phi)\,{\rm d}\phi,
\]
it is easy to prove that  ${\cal F}_m:{\cal H}^s\to W_m^s$ is just a right
inverse of $f\longmapsto f\otimes e_m$ and that 
\[
 \|{\cal F}_m F\|_{W_m^s}\le \|F\|_{{\cal H}^s}. 
\]
In particular, we have  
\begin{equation}
 \label{eq:Norm_decomposition}
 \|F\|_{{\cal H}^s}^2 = \sum_{m=-\infty}^\infty \|{\cal F}_m F\|^2_{W_m^s}. 
\end{equation}
In other words, $\{W_m^s\}_m$ gives rise to an orthogonal sum decomposition of
${\cal H}^s$ in its Fourier modes in the azimuthal angle $\phi$. 
Clearly, 
\[
 \|f \|_{W_m^0}^2= \|f \|_{L^2_{\sin}}^2:=\int_0^\pi
|f(\theta)|^2\sin\theta\,{\rm d}\theta,
\]
and therefore the space for $s=0$ is independent of $m$, provided that one
ignores how $f$ is extended outside of  $[0,\pi]$ (see 
\eqref{eq:periodicity-conditions}).
For $m=1,2$ it is possible to derive integral expressions for these norms,
as we show in the proof (in Appendix~\ref{sec:theorem3.1}) of the following technical result.  

\begin{theorem}
 \label{theo:equiv_norms}
Let  
\begin{eqnarray}
\|f\|_{{Z_0^1}}^2&:=& \frac14\int_0^{\pi
}|f(\theta)|^2\sin\theta\,{\rm
d\theta}+
 \int_0^{\pi }|f'(\theta)|^2{\sin\theta}\,{\rm d\theta},\label{eq:Y0}\\
\|f\|_{{Z_m^1}}^2&:=& m^2 \int_0^{\pi
}|f(\theta)|^2\frac{\rm
d\theta}{\sin\theta}+
 \int_0^{\pi }|f'(\theta)|^2{\sin\theta}\,{\rm d\theta},\label{eq:Ym}\\
\|f\|_{{Z_m^2}}^2&:=& m^4 \int_0^{\pi }|f(\theta)|^2\frac{\rm
d\theta}{\sin^3\theta}+
m^2\int_0^{\pi }|f'(\theta)|^2\frac{\rm
d\theta}{\sin\theta}+\int_0^{\pi }|f''(\theta)|^2{\sin\theta}\,{\rm
d\theta}.\label{eq:Zm}
\end{eqnarray}
Then, for all $m\in\Z$ with  $|m|\ge 1$,
\begin{equation}\label{eq:01:cor:equivNorms}
\|f\|_{W_0^1}=\|f\|_{Z_0^1}\le  \|f\|_{W_m^1}\le {\textstyle\frac{\sqrt{5}}2}
\|f\|_{{Z_m^1}}\le
{\textstyle\frac{\sqrt{5}}2}\|f\|_{W_m^1}. 
 \end{equation}
Moreover, for all  $m\in\Z$ with $|m|\ge 2$,
\begin{equation}\label{eq:02:cor:equivNorms}
{\textstyle\frac1{\sqrt{3}}}\|f\|_{W_m^2}\le \|f\|_{{Z_m^2}}\le  \sqrt{3}
\|f\|_{W_m^2}.
 \end{equation}
\end{theorem}
\begin{proof} See Appendix~\ref{sec:theorem3.1}.
\end{proof}

From this result, one can deduce that, for $m\ne 0$, $W_m^1\subset{\cal C}(\mathbb{R})$. Hence,  assume for simplicity that $f$ is a real valued function in $W_m^1=Z_m^1$, then it is easy to verify that $f^2, (f^2)' \in L^1_{\rm loc}(\mathbb{R})$. From the Sobolev embedding theorem one concludes that $f^2$, and therefore $f$, is a continuous function. It can  be seen next that necessarily, $f(0)=f(\pi)=0$,  since otherwise 
the first integral in the right hand side of \eqref{eq:Ym} could not be finite. This is no longer true for $m=0$ as it can easily seen by considering the counterexample 
$\left|\log |\sin\theta| \right|^{1/2}\in W_0^1$.

On the other hand,  from Theorem \ref{theo:equiv_norms} we obtain
\begin{equation}\label{eq:ineq:01}
 \|f\|_{W_m^1}\le \tfrac{\sqrt{5}}2 \|f\|_{W_{m+2n}^1},\qquad  \forall n \in \mathbb{N},
\end{equation}
and  
\begin{equation}\label{eq:ineq:02}
 \|f\|_{W_m^2}\le 3 \|f\|_{W_{m+2n}^2},\qquad   
   \forall n \in \mathbb{N}, \quad |m| \geq 2.
\end{equation}
We finish analyzing the regularity of
$W_m^s$ from a classical
Sobolev point of view. To this end, we introduce the $2\pi$-periodic Sobolev spaces
 \begin{equation}\label{eq:H_hash}
  H_{\#}^r:=\Big\{f\in H_{\rm loc}^r(\R)\ :  \ f=f(\cdot+2\pi)\Big\}
 \end{equation}
endowed with the norm
\begin{equation}
 \label{eq:SobolevNorm}
 \|f\|_{  H_{\#}^r}^2:=|\widehat{f}(0)|^2+\sum_{m\ne 0}| m|^{2r}
|\widehat{f}(m)|^2,\qquad
\widehat{f}(m)= \frac1{2\pi}\int_0^{2\pi} f(\theta)\exp(-{\rm i}m\theta)\,{\rm
d}\theta.
\end{equation} For $r=0$, $\|\cdot\|_{  H_{\#}^0}$
is, up to the factor $\sqrt{2\pi}$, the $L^2(0,2\pi)$ norm. For non-negative integer values of
$r$, an equivalent norm is
given by
\begin{equation}
 \label{eq:SobolevNorm:2}
\bigg[
\int_{0}^{2\pi}|f(\theta)|^2\,{\rm d}\theta+
\int_{0}^{2\pi}|f^{(r)}(\theta)|^2\,{\rm d}\theta\bigg]^{1/2}.
\end{equation}


 \begin{proposition} \label{prop:inclusion_in_Hr}
 For all $r>0$ there exists $C_r>0$ independent of $m$ and $f$ such that
 \begin{equation}
 \label{eq:01:prop:inclusion_in_Hr}
  \|f\|_{H^r_\#}\le C_r\|f\|_{W_{m}^{r+1/2}},\qquad \forall f\in
W_m^{r+1/2}.
 \end{equation}
Further,
 \begin{eqnarray}
 \label{eq:02:prop:inclusion_in_Hr}
  \|f\|_{W_m^0}&\le& \|f\|_{H^0_\#} 
\label{eq:03a:prop:inclusion_in_Hr},\quad 
 \|f\|_{W_m^1} \le 
C(1+|m|) \|f\|_{H^1_\#}
 \label{eq:03b:prop:inclusion_in_Hr},
 \end{eqnarray}
with $C$ independent of $f$ and $m$.
 \end{proposition}
 \begin{proof} See Appendix~\ref{sec:proposition3.2}.
 \end{proof}

%

\section{Error estimates  for spherical interpolatory approximations}\label{sec:err_est}
In this section we prove Theorem~\ref{theo:main} after deriving several associated one dimensional
interpolant properties. 
\subsection{Fourier analysis}
We consider the following  even and odd one dimensional interpolation problem, 
for the  data $f(\theta_j),~ j =0, \ldots N$:

\begin{eqnarray*}
\text{Find}~ \mathrm{q}_N^{\rm e} f \in \mathbb{D}_N^{\rm e}  f,\quad
\text{such that} \quad
 \mathrm{q}_N^{\rm e} f(\theta_j)&=&f(\theta_j),\quad j=0,\ldots,N, \\ \\
  \text{Find}~ \mathrm{q}_N^{\rm o} f \in \mathbb{D}_{N-2}^{\rm o} ,\quad
\text{such that} \quad
 \mathrm{q}_N^{\rm o} f(\theta_j)&=&f(\theta_j),\quad j=1,\ldots,N-1.
 \end{eqnarray*}
 For notational convenience, we introduce
\[
\mathrm{q}_N^{m}:=\begin{cases} 
 \mathrm{q}_N^{\rm e}, &\text{if $m$ is even},\\
 \mathrm{q}_N^{\rm o},  &\text{if $m$ is odd}.
 \end{cases}
\]
Then, as proved in~\cite[Section~4.1, Lemma 3]{DoGa:2012}, 
we have the following Fourier expression  connecting the matrix-free interpolant on
the sphere, defined in \eqref{eq:interpolant}, and the even and odd interpolants,  as a result of 
an aliasing process in $\phi$:  for all $F\in{\cal H}^r$ with $r>1$
\begin{equation}\label{eq:interp_new_rep}
\big({\cal Q}_N F\big)(\theta,\phi) =\sum_{-N+1\le  m\le
N} \big({\rm q}_N^{m} \rho_N^m
F\big)(\theta)e_{m}(\phi),\quad \text{with } \rho_N^m F:=\sum_{\ell=-\infty}^\infty   {\cal
F}_{m+2\ell
N}F. 
\end{equation} 
Then,  error in the spherical interpolatory approximation can be estimated 
as
\begin{equation}\label{eq:01:proof:MainTheo}
\begin{aligned}
 \|{\cal Q}_NF-F\|_{{\cal H}^s}^2&= \sum_{-N+1\le m\le N} 
 \|{\rm q}_N^{m} \rho_N^mF-{\cal F}_mF\|^2_{W_m^s}+\bigg[\sum_{m\ge  N+1}+\sum_{m\le -N} \bigg]
 \|{\cal F}_mF\|^2_{W_m^s}\\
 &\le  \sum_{-N+1\le m\le N} 
 \|{\rm q}_N^{m} {\cal F}_mF-{\cal F}_mF\|^2_{W_m^s}+\sum_{-N+1\le m\le N} 
 \|{\rm q}_N^{m} (\rho_N^mF-{\cal F}_mF)\|^2_{W_m^s} \\
 &  \qquad +\bigg[\sum_{m\ge  N+1}+\sum_{m\le -N} \bigg]
 \|{\cal F}_mF\|^2_{W_m^s} 
\end{aligned}
\end{equation}
Thus, for proving Theorem \ref{theo:main}, we have to bound three terms which depend
on the approximation properties of ${\rm q}_N^m$, the stability of this interpolant and the error introduced by ignoring the tail of the Fourier series (in $\theta$). The two first properties concerning for the  one-dimensional interpolant ${\rm q}_N^m$ will be explored in the next subsection.

\subsection{Error estimates for one dimensional interpolants}

Let ${\rm I}_N$ be the trigonometric interpolant for $2\pi$-periodic functions
defined by
\begin{equation}
 {\rm I}_N f \in {\rm span}\langle e_m\ :\ -N<m\le N\rangle,\quad
\text{such that }\quad {\rm I}_N g(\theta_j)=g(\theta_j),\quad j=-N+1,\ldots, N.
\label{eq:interTrig}
\end{equation}
Then, if we denote $g_-=g(-\:\cdot\:)$, it is easy to show that the average function
\[
 {\rm q}_Ng:=
{\textstyle\frac{1}{2}}\big[{\rm I}_N g +({\rm I}_N g_-)_-\big]\in {\rm
span}\langle e_m\ :\ -N\le m\le N\rangle \
\]
solves also \eqref{eq:interTrig} and preserves the parity of the integrand,
i.e., if $g$ is even/odd then so is ${\rm q}_Ng$. Further,
\[
  {\rm q}_Ng\in\mathbb{D}_N^{\rm e}\oplus\mathbb{D}_{N-2}^{\rm o}. 
\]
It is now straightforward to check that for $f_{\rm e}$ and $ f_{\rm o}$ $2\pi$-periodic 
even and  odd respectively
functions, we have
\[
 {\rm q}_N^{\rm e} f_{\rm e}= {\rm q}_N f_{\rm e},\quad 
 {\rm q}_N^{\rm o} f_{\rm o}= {\rm q}_N f_{\rm o}.
\]
These
relations and the well known Sobolev convergence estimates for $I_N$ (cf. \cite[Ch. 8]{SaVa:2002}), 
yield
\begin{equation}
\label{eq:qnmHs}
 \|\mathrm{q}_{N}^m f-f\|_{H^s_\#}\le C N^{s-t}\|f\|_{H^t_\#}, \qquad
 \text{for any $f\in H^t_{\#}$}\cap W_m^1, \qquad 0\le s\le t, \quad t>1/2.
\end{equation}

In the above inequality and in the reminder of this section, it is convenient to use $C$ to represent a generic positive constant that is independent of the truncation parameter $N$. 

In this subsection we will derive convergence estimates for $\mathrm{q}_{N}^m$ very similar to 
\eqref{eq:qnmHs} but with the norms $\|\cdot\|_{W_m^s}$ instead. We prove  such results for
the interpolant $\mathrm{q}_{N}^m$ in Theorem~\ref{theo:qNmf_1}, after developing ten 
auxiliary results in this subsection. To this end, we first start with inverse estimate:

\modification{
\begin{lemma}\label{lemma:inverseIneq} 
For any $s \geq 0$, there exists $C>0$ such that, for any
 $r_N\in\mathbb{D}_N^{\rm e}$, $s_N\in \mathbb{D}_{N
-2}^{\rm o}$,  the following estimate holds
\begin{equation}\label{eq:inv_ineq}
\|r_N\|_{W_0^s}\le C  N^{s}\|r_N\|_{W_0^0},\quad 
\|s_N\|_{W_1^s}\le C  N^{s}\|s_N\|_{W_1^0}.
\end{equation}
\end{lemma}
\begin{proof} The above inverse inequality for the case $s=1$  was established in
\cite[Lemma 6]{DoGa:2012} and the proof is similar for $s \geq 0$.
\end{proof}
}

\begin{proposition}\label{lemma:01}
For all  $m\in\mathbb{Z}$ and $N\ge 2$, there exist projections ${\rm p}_{N}^m$ on 
$\mathbb{D}_N^{\rm e}$ for even $m$ and on $\mathbb{D}_{N-2}^{\rm o}$ for odd $m$ which satisfy the following convergence estimate
\begin{equation}\label{eq:01:lemma:01}
 \|{\rm p}_N^m f-f\|_{W_m^s}\le C_{s,t} N^{s-t}\|f\|_{W_m^t},
\end{equation}
where $0\le s\le t$ with $t>1$ and $C_{s,t}$ independent of $f$. Moreover, 
\[
{\rm p}_N^m f(0)=f(0),\quad \text{and}\quad   {\rm p}_N^m f(\pi)=f(\pi).
\]
\end{proposition}
\begin{proof}
For $m\ne 0$ we can choose ${\rm p}_N^m $ to be the truncated partial sum
\[
{\rm T}_N^m f:=\sum_{0\le n\le N-1} \widehat{f}_m(n)Q_n^m\in{\rm
span}\langle Q_n^m \ :\ n\le N-1 \rangle
\subset\left\{\begin{array}{ll}
               \mathbb{D}_{N-1}^{\rm e},\quad&\text{if $m$ is even},\\
               \mathbb{D}_{N-2}^{\rm o},\quad&\text{if $m$ is odd},\\
              \end{array}
\right.
\]
with the choice of  $Q_n^m=0$ for $n<|m|$ in~\eqref{eq:normWmr} so that the sum above is
void for $n<|m|$.  The definition of
the norms of $W_m^s$ implies
\begin{equation}\label{eq:trun_est}
 \|{\rm T}_N^m f-f\|_{W_m^s}^2=\sum_{n\geq N} \left|n+\tfrac12\right|^{2s}
|\widehat{f}_m(n)|^2\le  \left(N+\tfrac12\right)^{2s-2t}\|f\|^2_{W_m^t}.
\end{equation}
We observe that \eqref{eq:01:lemma:01} holds actually for any $t\ge s$ and 
also that for $m\ne 0$, ${\rm p}_N^m f(0)={\rm
p}_N^m f(\pi)=0$ and $f\in W_m^1$ vanishes at $\{0,\pi\}$.

For $m=0$, we cannot ensure that
${\rm
T}_N^0 f(\{0,\pi\})=f(\{0,\pi\})$ which prompts us to consider a different projection.  In this case,
we choose ${\rm p}_N^0$ to be the interpolant
\[
 {\rm p}_N^0  \in \mathbb{D}_N^{\rm e}, \quad \text{such that} \quad 
\modification{p_N^0(\eta_j)}=f(\eta_j),\quad j=0,\ldots,N
\]
where $\{\cos\eta_j\}_j$ are the Gauss-Lobatto quadrature points,  which includes the 
endpoints, that is, $\eta_0=0$, and $\eta_N=\pi$. In
\cite[Appendix A, Proposition 6]{DoGa:2012} we proved that  for all $t>1$ there exists 
$C_t>0$ such that 
\begin{equation}\label{eq:pn0}
 \|{\rm p}_N^0 f-f\|_{W_0^0}\le C_t N^{-t}\|f\|_{W_0^t},\quad \forall f\in
W_0^t. 
\end{equation}
Using \eqref{eq:inv_ineq} and \eqref{eq:trun_est}, we first obtain
\begin{eqnarray*}
\|{\rm p}_N^0 f-f\|_{W_0^s}&\le& \|{\rm p}_N^0(f-{\rm T}_N^0
f)\|_{W_0^s}+\|{\rm T}_N^0f-  f\|_{W_0^s}\\
&\le& C N^{s}\|{\rm p}_N^0(f-{\rm
T}_N^0 f)\|_{W_0^0} + \left(N+\tfrac12\right)^{s-t}\|f\|_{W_0^t}\\
&\le& C \modification{N^{s}}\big[\|{\rm p}_N^0 f-f\|_{W_0^0}+\|{\rm
T}_N^0 f -f\|_{W_0^0}\big] + \left(N+\tfrac12\right)^{s-t}\|f\|_{W_0^t}\\
&\le& C \left[\left(N+\tfrac12\right)^{s}\|{\rm p}_N^0 f-f\|_{W_0^0}+2
\left(N+\tfrac12\right)^{s-t}\|f\|_{W_0^t}\right]
\end{eqnarray*}
and hence the desired result~\eqref{eq:01:lemma:01} follows by applying~\eqref{eq:pn0}
\end{proof}

 \begin{lemma}[{\cite[Proposition 4]{DoGa:2012}}]\label{lemma:ineqY1} 
 For $f \in Z_0^1$, there exists $C>0$  such that 
   \begin{eqnarray*}
    \|f\|_{H^0_\#}\le C\big[ 
    \|f\|_{L^2_{\sin}}+ \|f\|_{L^2_{\sin}}^{1/2} \|f'\|_{L^2_{\sin}}^{1/2} \big]
     \end{eqnarray*}
   \end{lemma}
 For the next result, we  introduce  $s_N, 
\widetilde{s}_N\in\mathbb{D}_N^{\rm e}$, the orthogonal and interpolating approximations of the $\sin(\cdot)$ function
on $[0, \pi]$. That is, 
\begin{subequations}
\label{eq:def:sn:sntilde}
\begin{eqnarray}
s_N &\in& \ \mathbb{D}_N^{\rm e}  \quad \text{such that }\quad
s_N(\theta_j)=\sin\theta_j,\quad j=0,\ldots, N,
\label{eq:def:sn}\\
\widetilde{s}_N  &\in & \mathbb{D}_N^{\rm e} \quad ~\text{such that }\quad \int_0^\pi
(\widetilde{s}_N(\theta)-\sin\theta)p_N(\theta)\,{\rm d}\theta=0,\quad  \forall
p_N \in\mathbb{D}_N^{\rm e}\label{eq:def:sntilde}.
\end{eqnarray}
\end{subequations}

\begin{lemma}\label{lemma:sin_approx} For all $N\ge 1$
\[
\|\widetilde{s}_N-\sin(\cdot)\|_{L^\infty(0,\pi)}\le \frac{2}{\pi N},\quad
\|s_N-\widetilde{s}_N\|_{L^\infty(0,\pi)} \le \frac{2}{\pi N}.
\]
\end{lemma}
\begin{proof} We  prove the result for even $N$. The odd $N$ case follows similarly. Straightforward
calculations show that for $\theta \in [0, \pi]$
\[
 \sin \theta=\frac{2}{\pi}-\frac{4}{\pi}\sum_{j=1}^\infty
\frac{\cos2j\theta}{4j^2-1}.
\]
On the other hand, we have the aliasing effect 
\[
 {\rm q}_{N}^{\rm e}\big(\cos((2j+2\ell N)\,\cdot\,)\big)=\cos(2j\,\cdot\,)=
 {\rm q}_{N}^{\rm e}\big(\cos((-2j+2\ell N)\,\cdot\,)\big),\quad  0\le
j\le N,\quad \forall \, \ell\in\mathbb{Z}.
\]
These properties imply that
\begin{eqnarray*}
 \widetilde{s}_N\modification{(\theta)}&=&\frac{2}{\pi}-\frac{4}{\pi}\sum_{j=1}^{N/2}
\frac{\cos2j\theta}{4j^2-1}\\
{s}_N(\theta)&=&\frac{2}{\pi}\bigg[1-\sum_{\ell=1}^\infty\frac{2}{(2\ell
N)^2-1} \bigg]\\
&&-\frac { 4 } { \pi } \sum_ { j=1 } ^{N/2-1}
\bigg[\frac{1}{4j^2-1}+
\sum_{\ell=1}^\infty\bigg( \frac{1}{(2j+2\ell
N)^2-1}+\frac{1}{(-2j+2\ell
N)^2-1}\bigg) \bigg]\cos 2j\theta\\
&& -\frac{4}{\pi}\bigg[\frac{1}{N^2-1}+ \sum_{\ell=1}^\infty
\frac{1}{(N+2\ell N)^2-1} \bigg]\cos  N\theta.
\end{eqnarray*}
Then, for any $\theta\in[0,\pi]$,
\begin{equation}\label{eq:sin-sntilde}
 |\widetilde{s}_N(\theta)-\sin \theta |\le \frac{4}{\pi}\sum_{j=N/2+1}^\infty
\frac{1}{4j^2-1}=\frac{2}{\pi  (N+1)}  \leq \frac{2}{\pi  N}.
\end{equation}
On the other hand, 
\begin{eqnarray*}
|\widetilde{s}_N(\theta)-s_N(\theta)|&\le&\frac{4}{\pi}\bigg[
\sum_{\ell=2}^\infty
\frac{1}{(\ell N)^2-1}\\
&&+\sum_{j=1}^{N/2-1}\sum_{ \ell=1 }^\infty\bigg(
\frac{1}{(2j+2\ell N)^2-1}+\frac{1}{((N-2j)+(2\ell-1) N)^2-1}\bigg)\bigg]
\\
&=&\frac{4}{\pi}\bigg[
\sum_{\ell=1}^\infty
\frac{1}{(N+\ell N)^2-1}+\sum_{j=1}^{N/2-1}\sum_{ \ell=1 }^\infty
\frac{1}{(2j+\ell N)^2-1}\bigg] \\
&=&\frac{4}{\pi} \sum_{j=1}^{N/2}\sum_{ \ell=1 }^\infty
\frac{1}{(2j+\ell N)^2-1}\le\frac{2}{\pi} \sum_{j=1}^{N}\sum_{ \ell=1 }^\infty
\frac{1}{(j+\ell N)^2-1}\\
&=&\frac{2}{\pi}\sum_{j=N+1}^\infty
\frac{1}{ j^2-1}=\frac{2}{\pi N}.
\end{eqnarray*}
\end{proof}

In order to prove the next result, it is convenient to consider the quadrature rules
\begin{eqnarray*}
 {\cal L}^1_{N}g&:=&\frac{\pi}{N}
 \sum_{k=0}^{N}\dprime g\big(\tfrac{k\pi}N\big),\quad \qquad 
 {\cal L}^2_{N}g:=\frac{\pi}{N}\sum_{j=1}^{N}
g({\textstyle\frac{(j-1/2)\pi} N}). 
\end{eqnarray*}
Since
\[
  {\cal L}^1_{N}\cos(m\cdot)=\left\{\begin{array}{ll}
                                     \pi,&\text{if $m=2\ell N$},\\
                                     0,&\text{otherwise},
                                    \end{array}\right.\quad
  {\cal L}^2_{N}\cos(m\cdot)=\left\{\begin{array}{ll}
                                     (-1)^{\ell}\pi,&\text{if $m=2\ell
N$},\\
                                     0,&\text{otherwise},
                                    \end{array}\right.\quad
\]
we easily deduce the equalities
\[
 {\cal L}_N^1g_N={\cal L}_N^2g_N=\int_{0}^\pi g_N(\theta)\,{\rm d}\theta,\qquad
\forall g_N\in
\mathbb{D}_{2N-1}^{\rm e}, 
\]
and therefore
\begin{equation}\label{eq:prop01:Ln}
 {\cal L}^1_N(  |g_N|^2)=
 {\cal L}^2_N ( |g_N|^2)= \|g_N\|_{H_\#^0}^2,\quad \forall g_N\in
\mathbb{D}_{N-1}^{\rm e}\cup \mathbb{D}_{N-1}^{\rm o}. 
\end{equation}
Moreover, 
\begin{equation}\label{eq:prop01b:Ln}
 {\cal L}^1_N ( |g_N|^2)\le 2 \|g_N\|_{H_\#^0}^2,\quad 
 {\cal L}^2_N(|g_N|^2)\le  \|g_N\|_{H_\#^0}^2, \quad \forall g_N\in
\mathbb{D}_N^{\rm e}. 
\end{equation}
Besides, for $g$ sufficiently smooth 
\begin{equation}\label{eq:prop02:Ln}
 {\cal L}_N^2g- {\cal L}_N^1 g=-4\sum_{\ell=0}^\infty
\int_0^{\pi} g(\theta)\cos(2N(1+2\ell) \theta)\,{\rm d}\theta.
\end{equation}
That is, the difference in the quadrature rules is 
four times the sum of the Fourier coefficients in the cosine series of 
order $2(1+2\ell) N$. 

The well known estimate for the error of the composite
rectangular rule
\begin{equation}\label{eq:QuadError}
\bigg|
 {\cal L}_N^1g-\int_0^\pi g(\theta)\,{\rm d}\theta\bigg|\le \pi N^{-1}\int_0^\pi |g'(\theta)|\,{\rm d}\theta
 \end{equation}
 will be used repeatedly in this section. Further it is useful to note the relation 
\begin{equation}\label{eq:L2n}
{\cal
L}^1_{2N}=\tfrac{1}2\big({\cal L}^1_N+{\cal L}^2_N\big).
\end{equation}
In  the proofs below we use the fact that \modification{if} $f\in W_m^1$, with $m\ne 0$, then
$f$ is continuous with $f(0)=f(\pi)=0$.


\begin{proposition}\label{prop:01}
Suppose that Hypothesis \ref{hypo:01} holds. Let $f\in W_m^1$ with $m\ne 0$. Then
\begin{equation}\label{eq:prop4.5_bound}
 \|\mathrm{q}_{N}^{m}f\|_{L^2_{\sin}}\le C\big[
 \|f\|_{L^2_{\sin}} + N^{-1} \|f'\|_{L^2_{\sin}} +
N^{-1/2}\|\mathrm{q}_{N}^mf\|_{H^0_{\#}}
 \big]
\end{equation}
with $C$ independent of $f$ and  $m$.
\end{proposition}
\begin{proof}
Since
\begin{equation}
 \label{eq:00:Prop:01}
  \|\mathrm{q}_{N}^mf\|_{L^2_{\sin}}^2 =
\int_0^\pi |\mathrm{q}_{N}^mf(\theta)|^2 s_N(\theta)\,{\rm d}\theta+
\int_0^\pi |\mathrm{q}_{N}^mf(\theta)|^2 (\sin\theta -
s_N(\theta))\,{\rm d}\theta, 
\end{equation}
using  Lemma \ref{lemma:sin_approx},
\begin{eqnarray*}
   \|\mathrm{q}_{N}^mf\|_{L^2_{\sin}}^2 &\le& \int_0^\pi
g_N(\theta)\,{\rm d}\theta+
CN^{-1}\|\mathrm{q}_{N}^mf\|_{H^0_{\#}}^2= {\cal L}^1_{2N}g_N+
CN^{-1}\|\mathrm{q}_{N}^mf\|_{H^0_{\#}}^2
 \label{eq:01:Prop:01}
\end{eqnarray*}
where $g_N:= |\mathrm{q}_{N}^mf|^2s_N\in \mathbb{D}_{3N}^{\rm e}$.
Applying \eqref{eq:L2n} we deduce the bound
\begin{equation}
 \label{eq:02:Prop:01}
   \|\mathrm{q}_{N}^mf\|_{L^2_{\sin}}^2\le 
   {\cal L}^1_{N}g_N+ {\textstyle\frac{1}{2}}\big( {\cal L}^2_Ng_N-{\cal
L}_N^1g_N\big)+ CN^{-1}\|\mathrm{q}_{N}^m f\|_{H^0_\#}. 
\end{equation}
Since  $f(0)=f(\pi)=0$ and 
$\mathrm{q}_N^mf(\theta_j)=
f(\theta_j)$, for all $j=0,\ldots, N$, using Lemma \ref{lemma:sin_approx} we easily deduce 
the following bound:
\begin{eqnarray}
 {\cal L}^1_Ng_N&=&\frac{\pi}{N}\sum_{j=1}^{N-1} |f(\theta_j)|^2
\sin\theta_j
+\frac{\pi}{N}\sum_{j=1}^{N-1}
|\mathrm{q}_N^{m}f(\theta_j)|^2(s_N(\theta_j)-\sin\theta_j)\nonumber\\
&\le& \frac{1}{N} \sum_{j=1}^{N-1} |f(\theta_j)|^2\sin\theta_j
+C N^{-1} {\cal L}^1_N(|\mathrm{q}_N^m f|^2)=: E_1+
 C N^{-1} {\cal L}^1_N(|\mathrm{q}_N^m f|^2).
 \label{eq:03:Prop:01}
\end{eqnarray}
Also, using~\eqref{eq:prop01b:Ln}, 
\begin{equation} \label{eq:03b:Prop:01}
 {\cal L}^1_N(|\mathrm{q}_N^m f|^2)\le 2\|\mathrm{q}_N^m f\|^2_{H_\#^0}.
\end{equation}
On the other hand,  $E_1$  in~\eqref{eq:03:Prop:01} can be bounded using the error of the rectangular rule \eqref{eq:QuadError}:
\begin{eqnarray}\label{eq:04:Prop:01}
 E_1&=& \int_0^\pi |f(\theta)|^2\sin\theta\,{\rm d}\theta + 
\bigg[ {\cal L}_N^1(|f|^2\sin(\cdot)) - \int_{0}^{\pi} |f(\theta)|^2\sin\theta\,{\rm d}\theta\bigg]\nonumber\\
&\le&\| f\|^2_{L^2_{\rm sin}}+\frac{\pi}{N}\int_0^{\pi} |(|f(\theta)|^2\sin\theta)'|\,{\rm d}\theta\nonumber\\
&\le& \| f\|^2_{L^2_{\rm sin}}+\frac{\pi}{N}\int_0^{\pi} |f(\theta)|^2\,{\rm d}\theta
+\frac{\pi}{N}\int_0^{\pi} |2f(\theta)f'(\theta)|\sin\theta\,{\rm d}\theta\nonumber \\
&\le& \| f\|^2_{L^2_{\rm sin}} + C N^{-1}\bigg[\|f\|^2_{L^2_{\rm sin}}+\|f\|_{L^2_{\rm sin}}\|f'\|_{L^2_{\rm sin}}+\int_0^{\pi} |f(\theta)f'(\theta)|\sin\theta\,{\rm d}\theta\bigg] \nonumber\\
&\le& \| f\|^2_{L^2_{\rm sin}} + C N^{-2}\|f'\|^2_{L^2_{\rm sin}}.
\end{eqnarray}
We stress that in \eqref{eq:04:Prop:01}  we have used Lemma \ref{lemma:ineqY1} and, in the last step, the inequality
\[
 \|f\|_{L^2_{\rm sin}}\|f'\|_{L^2_{\rm sin}}+\int_0^{\pi} |f(\theta)f'(\theta)|\sin\theta\le N \|f\|_{L^2_{\rm sin}}+ N^{-1}
 \|f'\|_{L^2_{\rm sin}}.
\]
Finally,  using the fact that $ |\mathrm{q}_N^m f|^2 s_N\in\mathbb{D}_{3N}^{\rm e}$ and
applying \eqref{eq:prop02:Ln} and 
Hypothesis \ref{hypo:01},  we deduce the bound
\begin{eqnarray}
 \frac{1}2\big( {\cal L}^2_Ng_N-{\cal
L}^1_Ng_N\big)&=&-2\int_0^\pi |\mathrm{q}_N^m f|^2(\theta)s_N(\theta)
\cos(2N\theta)\,{\rm d}\theta\nonumber\\
&=&-2\int_0^\pi |\mathrm{q}_N^m f|^2(\theta)\sin\theta
\cos(2N\theta)\,{\rm d}\theta\nonumber\\
&&-2\int_0^\pi |\mathrm{q}_N^m f|^2(\theta)(\sin\theta-s_{N}(\theta))
\cos(2N\theta)\,{\rm d}\theta\nonumber\\
&\le& \max\{c_{\mathrm{H}}^{(2)}, c_{\mathrm{H}}^{(3)}\}\int_0^\pi
|\mathrm{q}_N^mf|^2(\theta)\sin  \theta \,{\rm d}\theta + C N^{-1}
\|\mathrm{q}_N^mf\|^2_{H^0_\#},\label{eq:LN2-Ln1}
\end{eqnarray}
where we have used again Lemma \ref{lemma:sin_approx}. Collecting
\eqref{eq:03:Prop:01}-\eqref{eq:LN2-Ln1} in \eqref{eq:01:Prop:01}  (with $c_H:=\max\{c_H^{(2)}, c_H^{(3)}\}$) we conclude
\[
 (1-c_{\mathrm{H}})\|\mathrm{q}_{N}^mf\|_{L^2_{\sin}}^2\le C\big[
 \|f\|^2_{L^2_{\sin}} + N^{-2} \|f'\|^2_{L^2_{\sin}} +
N^{-1}\|\mathrm{q}_{N}^mf\|^2_{H^0_{\#}}
 \big].
\]
Hence the desired result \eqref{eq:prop4.5_bound} follows.
\end{proof}
\begin{corollary}\label{cor:01}
Suppose that Hypothesis \ref{hypo:01} holds.
For all $r>1$ there exists $C_r>0$ so that for all $f\in W_m^r$ [with 
$f(0)=f(\pi)=0$ for $m = 0$]
\begin{equation}\label{eq:cor4.6-bd}
 \|\mathrm{q}_{N}^m f\|_{W_{ m}^0}\le C\big[
 \|f\|_{W_{m}^0}+ N^{-1} \|f \|_{W_{ m}^1} + N^{-r} \|f \|_{W_{ m}^r}
 \big],
 \end{equation}
 with $C_r$ independent of $N$, $m$ and $f$.
\end{corollary}
\begin{proof} 
In light of Proposition \ref{prop:01}, we just have to bound 
\begin{equation}\label{eq:01:cor:01}
 N^{-1/2}\|{\rm q}_{N}^m f\|_{H^0_\#}\le  N^{-1/2}\|{\rm
q}_N^m f-f\|_{H^0_\#}+N^{-1/2}\|f\|_{H^0_\#}. 
\end{equation}
For the second term we can apply Lemma \ref{lemma:ineqY1} and  the
inequality $2ab\le N^{1/2}a+N^{-1/2}b$, to show that 
\begin{eqnarray}
 N^{-1/2}\|f\|_{H^0_\#}&\le&
C N^{-1/2}\big(\|f\|_{L^2_{\rm sin}}+\|f\|^{1/2}_{L^2_{\rm
sin}}\|f'\|^{1/2}_{L^2_{\rm sin}} \big)\le 
C\big[\|f\|_{L^2_{\rm sin}}+ N^{-1}\|f'\|_{L^2_{\rm sin}}\big].
 \label{eq:02:cor:01}
\end{eqnarray}
On the other hand, \eqref{eq:qnmHs} implies that for all $r>1$ there exists
$C_r$ so that
\begin{equation}
 N^{-1/2}\| \mathrm{q}_N^{m}f-f\|_{H^0_\#}\le C_rN^{-r}\|f\|_{H^{r-1/2}_\#}\le
C_r'N^{-r}\|f\|_{W_{m}^r}, \label{eq:03:cor:01}
\end{equation}
where we have applied in the last step Proposition \ref{prop:inclusion_in_Hr}.
We note that $C_r$ is again independent of $m$. Applying \eqref{eq:02:cor:01} and
\eqref{eq:03:cor:01} in \eqref{eq:01:cor:01} we deduce the bound
\[
 N^{-1/2}\|{\rm q}_{N}^m f\|_{H^0_\#}\le C_r\big[
\|f\|_{H^0_\#}+N^{-1}\|f\|_{W_{ m}^{1}}+N^{-r}\|f\|_{W_{ m}^{r}}\big]
\]
and hence the desired results~\eqref{eq:cor4.6-bd} follows.
\end{proof}
%
%

To prove stability estimates in $W_m^1$, it is convenient to introduce the notation
\[
 \|f\|_{L^2_{\sin^{r}}}^2:=\int_0^{\pi}|f(\theta)|^2\sin^r\theta\,{\rm d}\theta,
\]
(in particular with $r=1, -1,-2$ or $-3$) and use the equivalence of
norms described in Theorem \ref{theo:equiv_norms} involving the two terms
 \begin{equation}\label{eq:terms}
 |m|\|{\rm q}_{N}^{m}f\|_{L^2_{\sin^{-1}}} \qquad \text{and} \qquad \|({\rm
q}_{N}^{m}f)'\|_{L^2_{\rm sin}}.
\end{equation}
The second term  is easily controlled by using inverse inequality and the results developed so far.

\begin{lemma} \label{lemma:02}
Suppose that Hypothesis \ref{hypo:01} holds. There exists $C>0$ independent of $m$, $f$ and $N$ such that 
\begin{equation}
\label{eq:01:lemma:02}
  \|({\rm q}_{N}^mf)'\|_{L^2_{\rm sin}} \le C N^{-1} \big[
 \|f\|_{W_{m}^0}+ N^{-1} \|f \|_{W_{ m}^1} + N^{-r} \|f \|_{W_{ m}^r}
 \big].
\end{equation}
\end{lemma}
\begin{proof}  Lemma \ref{lemma:inverseIneq} and the equivalent
norms presented in Theorem \ref{theo:equiv_norms} allow us to conclude, as a byproduct, the
inverse inequality
\begin{equation}\label{eq:wm1:01}
 \|({\rm q}_{N}^mf)'\|_{L^2_{\rm sin}} \le C N\|{\rm
q}_{N}^m f\|_{L^2_{\rm sin}}.
\end{equation}
Hence applying Corollary~\ref{cor:01} will lead to the derivation of the bound~\eqref{eq:01:lemma:02}.
\end{proof}
 

%
%

The analysis of the first term in  \eqref{eq:terms} is rather more delicate. Thus, before entering in the
analysis we need to prove some technical results. 

\begin{lemma}\label{lemma:aux5}
There exists $C>0$ so that for all $r_N\in\mathbb{D}_{N-2}^{\rm e}$ with $N\ge 2$, 
\begin{equation}\label{eq:lem4.8-bd}
 |r_N(0)|^2+|r_N(\pi)|^2 \le C N^2\log N\,{\cal L}_N^1(|r_N|^2\sin(\cdot) ).
\end{equation}
\end{lemma}
\begin{proof}
For each $j=1,\ldots,N-1$ define
\[
 L_j^N:=(-1)^{j+1}\sin^2\theta_j \frac{\sin N \theta}{N \sin\theta (\cos\theta-\cos\theta_j) }\in\mathbb{D}_{N-2}^{\rm e}.
\]
It is a simple exercise to verify that
\[
 L_j^N(\theta_i)=\left\{\begin{array}{ll}
                    1,\quad&i=j,\\
                    0,\quad&\text{otherwise}. 
                   \end{array}\right.
\]
Thus, $\{L_j^N\}$ is the Lagrange basis for the interpolation problem on $\mathbb{D}_{N-2}^{\rm e}$
with grid points $\{\theta_j\}_{j=1}^{N-1}$. \modificationG{Consequently,}
\[
 r_N=\sum_{j=1}^{N-1}r_N(\theta_j)L_j^N. 
\]
Since
\[
 L_j^N(0)= (-1)^{j+1}2\cos^2(\theta_j/2),
\]
we obtain
\begin{eqnarray}
 | r_N(0)|^2&\le&\bigg[\sum_{j=1}^{N-1}  2|r_N(\theta_j)|\bigg]^2\le \frac{2 N}{\pi}
 \bigg[\sum_{j=1}^{N-1}\sin^{-1}\theta_j\bigg]\bigg[\frac{\pi}{N}\sum_{j=1}^{N-1} |r_N(\theta_j)|^2\sin\theta_j\bigg]\\
 &\le&\frac{4 N^2}{\pi}\bigg[\sum_{1\le j\le N/2} \frac{1}{j}\bigg]{\cal
L}_N^1(|r_N|^2\sin(\cdot)), 
\end{eqnarray}
where we have used the inequality
\[
  \sin\theta_{N-j}= \sin\theta_j\ge \frac{2\theta_j}{\pi}=\frac{2j}{N},\quad \forall  j=0,\ldots, \lfloor N/2\rfloor. 
\]
Hence, for $N\ge 2$, the desired result~\eqref{eq:lem4.8-bd} for $r_N(0)$  follows from the inequality
\[
 \sum_{1\le j\le N/2} \frac{1}{j}\le 2\log N.  
\]
The bound for  $r_N(\pi)$ in~\eqref{eq:lem4.8-bd}  can be established  analogously. 
\end{proof}

Next we establish bounds for ${\rm q}_N^m$, by investigating
separately the cases for even $m$ (i.e., operator ${\rm q}_N^{\rm e}$) and for odd $m$
(i.e., operator ${\rm q}_N^{\rm o}$). 

\begin{proposition}
\label{prop:02b} Suppose that Hypothesis \ref{hypo:01} holds.  There exists $C>0$ such that for any integer $m\ne 0$ and $f\in
W_{2m}^2$,  
\begin{equation}\label{eq:prop4.9-bd}
|2m|\:\|{\rm q}_{N}^{\rm e}f\|_{L^2_{\sin^{-1}}}\le
C\big[\|f\|_{W_{2m}^1}+N^{-1}\|f\|_{W_{2m}^2}\big].
\end{equation}
\end{proposition}
\begin{proof}  We assume throughout this proof that $f$ is a real valued function. 
We consider the function  $g_{2N}(\theta):=|({\rm q}_{N}^{\rm e}
f(\theta))|^2/\sin^2(\theta)s_{2N}(\theta)\in\mathbb{D}_{4N-4}^{\rm e}$. Using the definition of
$\tilde{s}_{2N}$ in \eqref{eq:def:sn:sntilde} and Lemma  \ref{lemma:sin_approx}, 
\begin{eqnarray}
\|{\rm q}_{N}^{\rm e}f\|^2_{L^2_{\sin^{-1}}}&=&
\int_0^\pi  \Big|\frac{{\rm q}_{N}^{\rm e}f}{\sin\theta}\Big|^2
\widetilde{s}_{2N}(\theta)\,{\rm d}\theta
=\int_0^\pi g_{2N}(\theta)\,{\rm d}\theta+ \int_0^\pi  \Big|\frac{{\rm
q}_{N}^{\rm e}f}{\sin\theta}\Big|^2
(\widetilde{s}_{2N}(\theta)-{s}_{2N}(\theta))\,{\rm d}\theta\nonumber\\
&\le&{\cal L}_{2N}^1 g_{2N}+\frac{1}{\pi N}\|{\rm q}_{N}^{\rm
e}f\|^2_{L^2_{\sin^{-2}}}\nonumber\\
&=&{\cal L}_{N}^1 g_{2N}+{\textstyle\frac12}({\cal L}_{N}^2 g_{2N}-{\cal
L}_{N}^1 g_{2N})+\frac{1}{\pi N}\|{\rm q}_{N}^{\rm
e}f\|^2_{L^2_{\sin^{-2}}}.\label{eq:04:prop:02b}
\end{eqnarray}
Proceeding similarly as in \eqref{eq:LN2-Ln1},
using  Hypothesis \ref{hypo:01} (for ${\rm q}_{N}^{\rm e}
f(\theta) /{\sin \theta }\in\mathbb{D}_{N-2}^{\rm o}$) and again Lemma \ref{lemma:sin_approx}, 
we obtain
\begin{equation} \label{eq:05:prop:02b}
\begin{split}
{\textstyle\frac12}({\cal L}_{N}^2 g_{2N} & 
-{\cal L}_{N}^1 g_{2N}) = -2\int_0^\pi g_N(\theta)\cos(2N\theta)\,{\rm
d}\theta\\
&=\ -2\int_0^\pi  \bigg|\frac{{\rm q}_{N}^{\rm e}
f(\theta)}{\sin \theta }\bigg|^2 \sin\theta\cos 2N\theta\, {\rm d}\theta -2\int_0^\pi 
  \bigg|\frac{{\rm q}_{N}^{\rm e}
f(\theta)}{\sin \theta }\bigg|^2 (\widetilde{s}_{2N}(\theta)-\sin\theta) \cos(2N\theta)\,{\rm
d}\theta\\
&\le\  
c_{\rm H}^{(2)}\int_0^\pi  \bigg|\frac{{\rm q}_{N}^{\rm e}
f(\theta)}{\sin \theta }\bigg|^2 \sin\theta\,{\rm d}\theta + 
\frac{2}{\pi N}\|{\rm q}_{N}^{\rm e}f\|^2_{L^2_{\sin^{-2}}}
\end{split}
\end{equation}
Using \eqref{eq:05:prop:02b} in \eqref{eq:04:prop:02b} and  the identity
\[
 {\cal L}_{N}^1\big(|{\rm q}_{N}^{\rm e}f/\sin(\cdot)
\big|^2\big)=\|{\rm q}_{N}^{\rm e}f\|^2_{L^2_{\sin^{-2}}},
\]
we easily derive the bound
\begin{equation}\label{eq:03:ConvW1}
 \|{\rm q}_{N}^{\rm e}f\|^2_{L^2_{\sin^{-1}}}\le \frac{1}{1-c_{\rm H}^{(2)}}\Big[
 {\cal L}_{N}^1 g_{2N} 
+\frac{3}{\pi N}{\cal L}_{N}^1\big(|{\rm q}_{N}^{\rm e}f/\sin(\cdot)
\big|^2\big)
 \Big].
\end{equation}
The first term in the above bound can be estimated as follows.
Using the definition of ${s}_{2N}$, \eqref{eq:QuadError} and  the 
Cauchy-Schwarz inequality (combined with 
the inequality $2ab\le N a^2+N^{-1}b^2$), we obtain
\begin{eqnarray}
{\cal L}_{N}^1 g_{2N}&=&\frac{\pi}{N}\sum_{j=1}^{N-1}
\frac{|{\rm q}_{N}^{\rm e}
f(\theta_j)|^2}{\sin^2 \theta_j} s_{2N}(\theta_j)=
\frac{\pi}{N}\sum_{j=1}^{N-1}
\frac{|{\rm q}_{N}^{\rm e}
f(\theta_j)|^2}{\sin \theta_j}=\frac{\pi}{N}\sum_{j=1}^{N-1}
\frac{|f(\theta_j)|^2}{\sin \theta_j}\nonumber\\
&\le &\int_0^\pi |f(\theta)|^2\frac{\rm d\theta}{\sin\theta}
+\frac{\pi}{N}\bigg[\int_0^\pi |f(\theta)|^2\frac{\rm d\theta}{\sin^2\theta}
+\int_0^\pi 2|f(\theta)f'(\theta)|\frac{\rm
d\theta}{\sin\theta}\bigg]\nonumber\\
&\le& \Big(1+\frac{3\pi}2\Big)\|f\|^2_{L^2_{\sin^{-1}}} + \frac{\pi}{2N^2} \big(
\|f\|^2_{L^2_{\sin^{-3}}}+2\|f'\|^2_{L^2_{\sin^{-1}}}\big).   \label{eq:04:ConvW1}
\end{eqnarray}

\modification{On the other hand since $({\rm q}_{N}^{\rm e}f)(0)=f(0)=0=f(\pi)=({\rm
q}_{N}^{\rm e}f)(\pi)$ and using the fact that ${\rm q}_{N}^{\rm e}f\in\mathbb{D}_{N}^{\rm
e}$, we obtain
\[
({\rm q}_{N}^{\rm e}f/\sin)(0)=({\rm q}_{N}^{\rm e}f/\sin)(\pi)=0 .
\]
Hence
}
\begin{eqnarray}
{\cal L}_{N}^1\Big({\rm q}_{N}^{\rm e}f/\sin(\cdot) \Big)^2 &=&
\frac{\pi}{N}\sum_{j=1}^{N-1} \bigg|\frac{{\rm q}_{N}^{\rm e}f(\theta_j)}{\sin
\theta_j}\bigg|^2
=\frac{\pi}{N}\sum_{j=1}^{N-1} \bigg|\frac{f(\theta_j)}{\sin \theta_j}\bigg|^2\nonumber\\
&\le&\int_0^\pi \frac{|f(\theta)|^2}{\sin^2\theta} {\rm d\theta}
+\frac{\pi}{N}\bigg[2\int_0^\pi |f(\theta)|^2\frac{\rm d\theta}{\sin^3\theta}
+\int_0^\pi 2f(\theta)f'(\theta)\frac{\rm
d\theta}{\sin^2\theta}\bigg]\nonumber\\
&\le &  N\Big[\big(\pi+\tfrac12\big)\|f\|^2_{L^2_{\sin^{-1}}} + N^{-2}\big\{(3\pi+\tfrac12)
\|f\|^2_{L^2_{\sin^{-3}}}+\pi \|f'\|^2_{L^2_{\sin^{-1}}}\big\}\Big] 
\nonumber. 
\end{eqnarray}
In other words, 
\begin{equation}\label{eq:05:ConvW1}
N^{-1}{\cal L}_{N}^1\big(|{\rm q}_{N}^{\rm e}f/\sin(\cdot) \big|^2\Big)\le 
C\Big[ \|f\|^2_{L^2_{\sin^{-1}}}+N^{-2} \Big(\|f'\|^2_{L^2_{\sin^{-1}}} + \|f\|^2_{L^2_{\sin^{-3}}}\Big)\Big].
\end{equation}
Plugging \eqref{eq:04:ConvW1} and \eqref{eq:05:ConvW1} in \eqref{eq:03:ConvW1}, and taking into account the definitions of the equivalent norms $\|\cdot\|_{Z_m^1}$ and $\|\cdot\|_{Z_m^2}$ (see Theorem \ref{theo:equiv_norms}), we obtain the desired result~\eqref{eq:prop4.9-bd}.
\end{proof}

For the next result, we recall that the equivalence norms relation in~\eqref{eq:02:cor:equivNorms}  is valid only for $|m| \geq 2$. 

\begin{proposition}
\label{prop:02c} 
Suppose that Hypothesis \ref{hypo:01} holds. 
There exists $C>0$ such that for any $f\in
W_{2m+1}^2$ with $m\ne -1,0$, 
\begin{equation}\label{eq:prop4.10-bd}
|2m+1|\:\|{\rm q}_{N}^{\rm o}f\|_{L^2_{\sin^{-1}}}\le
(C+C'\sqrt{\log N})\big[\|f\|_{W_{2m+1}^1}+N^{-1}\|f\|_{W_{2m+1}^2}\big].
\end{equation}
\end{proposition}
\begin{proof}
Following the same steps as in the proof of Proposition \ref{prop:02b},
 using Hypothesis 1 again with $j=1$ (because ${\rm q}_{N}^{\rm o}f(\theta)/\sin(\theta)\in\mathbb{D}_{N-2}^{\rm e}$),  we obtain  
\begin{equation}
 \|{\rm q}_{N}^{\rm o}f\|^2_{L^2_{\sin^{-1}}}\le \frac{1}{1-c_{\rm H}^{(1)}} \Big[
{\cal L}_{N}^1 g_{2N} 
+\frac{3}{\pi N}{\cal L}_{N}^1\Big({\rm q}_{N}^{\rm o}f/\sin(\cdot) \Big)^2 \
\Big].\label{eq:03:ConvW1b}
\end{equation}
 First term can be treated as in \eqref{eq:04:ConvW1} to get
 \begin{eqnarray}
{\cal L}_{N}^1 g_{2N}  
&\le& C\Big[\|f\|^2_{L^2_{\sin^{-1}}} + N^{-2}\big(
\|f\|^2_{L^2_{\sin^{-3}}}+\|f'\|^2_{L^2_{\sin^{-1}}}\big)\Big].
\label{eq:quad:ruleb}
 \end{eqnarray}
 The main difference compared with  the  even $2m$ case dealt in the previous Proposition 
 arises in the second term, since we now have 
 \begin{eqnarray*}
 {\cal L}_{N}^1\Big({\rm q}_{N}^{\rm o}f/\sin(\cdot) \Big)^2&=&\frac{\pi}{2N} 
 \bigg[\Big|\Big(\frac{{\rm q}_{N}^{\rm o}f}{\sin}\Big)(0)\Big|^2+
 \Big|\Big(\frac{{\rm q}_{N}^{\rm o}f}{\sin}\Big)(\pi)\Big|^2\bigg]
 +\frac{\pi}{N}\sum_{j=1}^{N-1}
\frac{|f(\theta_j)|^2}{\sin^2 \theta_j}\\
 &=:&S_1+S_2. 
 \end{eqnarray*}
The first term $S_1$ did not appear in the proof of Proposition \ref{prop:02b}, 
since ${\rm q}_{N}^{\rm o}f\in\mathbb{D}_{N-2}^{\rm o}$, or, equivalently, 
${\rm q}_{N}^{\rm o}f/\sin\in \mathbb{D}_{N-2}^{\rm e}$. Thus, we can expect
  $({\rm q}_{N}^{\rm o}f/\sin\big)(0), ({\rm q}_{N}^{\rm o}f/\sin\big)(\pi)\ne 0$.

Clearly, the second term can be bounded as in \eqref{eq:05:ConvW1}:
 \begin{equation}\label{eq:05:ConvW1b}
N^{-1} S_2\le C\Big[ \|f\|^2_{L^2_{\sin^{-1}}}+N^{-2} \left(\|f'\|^2_{L^2_{\sin^{-1}}} +\|f\|^2_{L^2_{\sin^{-3}}}\right)\Big]. 
\end{equation}
For $S_1$ we apply Lemma \ref{lemma:aux5} and we follow  arguments similar to the
derivation of  \eqref{eq:04:ConvW1} to obtain 
\begin{eqnarray}
 N^{-1}S_1&\le& C \log N\: {\cal L}_N^1 \big(|{\rm q}_{N}^{\rm
o}f|^2/\sin(\cdot)\big)=C\log N\bigg[
 \frac{\pi}{N}\sum_{j=1}^{N-1} 
 \frac{|f(\theta_j)|^2}{\sin \theta_j}\bigg]\nonumber\\
 &\le& C'\log N\Big[\|f\|^2_{L^2_{\sin^{-1}}}+N^{-2}\|f'\|^2_{L^2_{{\sin^{-1}}}}+ N^{-2}\|f\|^2_{L^2_{{\sin^{-3}}
 }}\Big].\label{eq:05:ConvW1c}
\end{eqnarray}
Thus we have proved  the inequality
\[
  \|{\rm q}_{N}^{\rm o}f\|^2_{L^2_{\sin^{-1}}}\le 
 ( C_1+C_2\log N)\Big[ \|f\|^2_{L^2_{\sin^{-1}}}+N^{-2} \|f'\|^2_{L^2_{\sin^{-1}}} +N^{-2}\|f\|^2_{L^2_{\sin^{-3}}}\Big]. 
\]
The desired result~\eqref{eq:prop4.10-bd} now follows from  Theorem~\ref{theo:equiv_norms}.
\end{proof}

Now we are ready to establish convergence estimates for $\mathrm{q}_{N}^m$  similar to 
\eqref{eq:qnmHs} in $\|\cdot\|_{W_m^s}$ norms.

\begin{theorem}\label{theo:qNmf_1} 
Suppose that Hypothesis \ref{hypo:01} holds. There exists $C>0$ so that for any $f\in W_{m}^{2}$,  
\begin{subequations}\label{eq:01:theo:qNmf_1} 
\begin{eqnarray}
  \|\mathrm{q}_{N}^m f\|_{W_{ m}^0}&\le& C\big[
 \|f\|_{W_{m}^0}+ N^{-1} \|f \|_{W_{ m}^1} + N^{-2} \|f \|_{W_{ m}^2}
 \big],\label{eq:01a:theo:qNmf_1} \\
 \|\mathrm{q}_N^{2m} f\|_{W_{2m}^1}&\le& C \Big[N\| f\|_{W_{2m}^0}+\|
f\|_{W_{2m}^1}+
 N^{-1}\| f\|_{W_{2m}^2}\Big],\label{eq:01b:theo:qNmf_1}  \\
  \|\mathrm{q}_N^{  2m+1} f\|_{W_{2m+1}^1}&\le& C(1+\sqrt{\log N}) \Big[N\|
f\|_{W_{2m+1}^0}+\| f\|_{W_{2m+1}^1}+
 N^{-1}\| f\|_{W_{2m+1}^2}\Big].\quad \label{eq:01c:theo:qNmf_1}  
\end{eqnarray}
\end{subequations}
Moreover, for all $r\ge 2$ there exists $C_r>0$ so that for all $m$
\begin{subequations}\label{eq:02:theo:qNmf_1} 
\begin{eqnarray}
\|\mathrm{q}_N^{m} f-f\|_{W_m^0}&\le& \modificationG{C_r}
N^{-r}\|f\|_{W_{m}^r},\label{eq:02a:theo:qNmf_1} \\
\|\mathrm{q}_N^{2m} f-f\|_{W_{2m}^1}&\le& \modificationG{C_r}
N^{1-r}\|f\|_{W_{2m}^r},\label{eq:02b:theo:qNmf_1}\\ 
\|{\mathrm{q}}_N^{2m+1} f-f\|_{W_{2m+1}^1}&\le& \modificationG{C_r}(1 +\sqrt{\log N})
N^{1-r}\|f\|_{W_{2m+1}^r}.\label{eq:02c:theo:qNmf_1} 
\end{eqnarray}
\end{subequations}
\end{theorem}
\begin{proof} 
\modificationG{We recall that for $m\ne 0$,  if $f\in W_m^2$, then $f$ vanishes at $0,\pi$.  
We first consider the case $m\ne 0$: Corollary \ref{cor:01}  yields the bounds  \eqref{eq:01a:theo:qNmf_1};  the bound 
 \eqref{eq:01b:theo:qNmf_1} follows from Lemma~\ref{lemma:02} and  Proposition \ref{prop:02b}; and \eqref{eq:01c:theo:qNmf_1} is a consequence of  Lemma~\ref{lemma:02}  and Proposition  \ref{prop:02c}. The latter conclusion applies also for the case $m \ne -1$.
For $m= 0,-1$, \eqref{eq:01c:theo:qNmf_1} can be deduced similarly, via the inverse inequalities applied to estimate \eqref{eq:01a:theo:qNmf_1}. 
For $m=0$, and under the additional assumption that $f(0)=f(\pi)=0$, 
the  bound \eqref{eq:01a:theo:qNmf_1} was also established in Corollary \ref{cor:01}.
The estimate \eqref{eq:01b:theo:qNmf_1} follows from combining \eqref{eq:01a:theo:qNmf_1} and the inverse inequalities stated in Lemma \ref{lemma:inverseIneq}. 
}

With ${\mathrm{p}}_N^m$ being  the projection introduced in Proposition~\ref{lemma:01}, we 
observe that 
\[
\|\mathrm{q}_{N}^{m}f-f\|_{L^2_{\rm sin}}\le
\|\mathrm{q}_{N}^{m}(f-\mathrm{p}_{N}^mf)\|_{L^2_{\rm sin}}+
\|\mathrm{p}_{N}^{m}f-f\|_{L^2_{\rm sin}}.
\] 
Further, $\big(f-\mathrm{p}_N^{m}f\big)(0)=\big(f-\mathrm{p}_N^{m}f\big)(\pi)=0$ even
for $m=0$. Thus Corollary \ref{cor:01} (or \eqref{eq:01:theo:qNmf_1} in the cases proven
up to now) can be applied to derive the bound  
\begin{eqnarray*}
 \|\mathrm{q}_{N}^{m}f-f\|_{L^2_{\rm sin}}&\le& \modification{C}
\big[ \|f-\mathrm{p}_N^{m}f\|_{L^2_{\rm
sin}}+N^{-1}\|f-\mathrm{p}_N^{ m}f\|_{W_m^1}+
N^{-2}\|f-\mathrm{p}_N^{ m}f\|_{W_m^2}\big],
\end{eqnarray*}
where $C$ is independent of $m$ and $f$. Now~\eqref{eq:02a:theo:qNmf_1} 
follows from  Proposition \ref{lemma:01}. 

To prove \eqref{eq:02b:theo:qNmf_1},  we proceed as before,
using Lemma~\ref{lemma:02}  and Proposition \ref{prop:02b},   to  obtain the inequality
\[
  \|\mathrm{q}_{N}^{2m}f-f\|_{W_{2m}^1}\le C\Big[N
\|f-\mathrm{p}_N^{2m}f\|_{W_{2m}^0}+
   \|f-\mathrm{p}_N^{2m}f\|_{W_{2m}^1}+N^{-1}
\|f-\mathrm{p}_N^{2m}f\|_{W_m^2}\Big].
\]
Hence Proposition \ref{lemma:01} yields the estimate~\eqref{eq:02b:theo:qNmf_1}.

The procedure for proving estimate \eqref{eq:02c:theo:qNmf_1} is completely analogous. 
We note that \eqref{eq:01a:theo:qNmf_1}-\eqref{eq:01b:theo:qNmf_1} for $m=0$ in the general case  (i.e., for functions not vanishing at $\{0,\pi\}$)
can be now deduced from \eqref{eq:02a:theo:qNmf_1}-\eqref{eq:02b:theo:qNmf_1}.
\end{proof}

Next we are ready conclude  Section~\ref{sec:err_est}  by proving the main 
spectrally accurate convergence result of this  article, namely,  Theorem \ref{theo:main}. 

\subsection{Proof of Theorem \ref{theo:main}}
The proof starts from \eqref{eq:01:proof:MainTheo} where we 
have derived  
\[
  \|{\cal Q}_NF-F\|_{{\cal H}^s}^2\le E_1+E_2+E_3 
\]
with
\begin{eqnarray}
 E_1&:=&\sum_{-N+1\le m\le N} 
 \|{\rm q}_N^{m} {\cal F}_mF-{\cal F}_mF\|^2_{W_m^s}\\
 E_2&:=&\sum_{-N+1\le m\le N} 
 \|{\rm q}_N^{m} (\rho_N^mF-{\cal F}_mF)\|^2_{W_m^s} \\
E_3&=&\bigg[\sum_{m\ge  N+1}+\sum_{m\le -N} \bigg]
 \|{\cal F}_mF\|^2_{W_m^s} 
\end{eqnarray}
For the sake of simplicity we can restrict ourselves to consider only $F\in {\rm span}\:\{Y_n^m\}$ which makes the sums above to be finite. The general case can be deduced by a density argument. Moreover, we can 
take $s\in\{0,1\}$ since the result for intermediate values of $s$ follows from the theory of 
interpolation of Sobolev spaces~\cite{McLean:2000}.

The third term can be estimated with the help of \eqref{eq:inequatityWms} and \eqref{eq:Norm_decomposition}, as follows: For $t\ge s$,  
\begin{equation}\label{eq:E3}
 E_3 \le \sum_{|m|\ge  N} \left(|m|+\tfrac{1}2\right)^{2s -2t}
 \|{\cal F}_mF\|^2_{W_m^t} \le \sum_{|m|\ge  N}  N^{2s-2t}
 \|{\cal F}_mF\|^2_{W_m^t}\le N^{2(s-t)}\|F\|_{{\cal H}^t}^2.
\end{equation}
For $E_1$, we \modification{apply} \eqref{eq:02:theo:qNmf_1} to obtain
\begin{eqnarray}
 E_1\le C(1+\sqrt{\log N})^{s} N^{2s-2t} \sum_{-N+1\le m\le N} \|{\cal F}_m F\|_{W_m^t}^2\le
 C(1+\sqrt{\log N})^{s} N^{2s-2t}\|F\|_{{\cal H}^t}^2. \label{eq:E1}
\end{eqnarray}
Regarding $E_2$, we apply  \modification{the definition of $\rho_N^m$ in \eqref{eq:interp_new_rep} and estimates \eqref{eq:01:theo:qNmf_1} of Theorem \ref{theo:qNmf_1} to obtain first
\begin{eqnarray}
E_2&\le&  \sum_{-N+1\le m\le N} \bigg[\sum_{\ell\ne 0}
 \|{\rm q}_N^{m}{\cal F}_{m+2\ell N} F\|_{W_{m}^s}\bigg]^2\nonumber\\
 &\le& C(1+\log N)^s N^{2s}\sum_{-N+1\le m\le N} \sum_{j=0}^2 N^{-2j}\bigg[\sum_{\ell\ne 0}  
 \| {\cal F}_{m+2\ell N} F\|_{W_{m}^j}
 \bigg]^2. \label{eq:new00}
 \end{eqnarray}
Let us study now the three terms in the last sum above. Cauchy-Schwarz inequality and \eqref{eq:ineq:01}-\eqref{eq:ineq:02} leads to
\begin{eqnarray}
\bigg[\sum_{\ell\ne 0}  
 \| {\cal F}_{m+2\ell N} F\|_{W_{m}^j}
 \bigg]^2&\le& 9\bigg[\sum_{\ell\ne 0}
 \frac{1}{|m+2\ell N|^{2t-2j}}\bigg]\bigg[\sum_{\ell\ne 0} |m+2\ell N|^{2t-2j}
 \|{\cal F}_{m+2\ell N} F\|_{W_{m+2\ell N}^j}^2\bigg]\qquad 
 \label{eq:new01}
\end{eqnarray}
for $j=0,1,2$.  Since
\[
 \sum_{\ell \ne 0}\frac{1}{|x+\ell|^{r}}\le C_r,\quad \forall x\in[-1/2,1/2]
\]
with $C_r$ depending only $r>1$, we can bound \eqref{eq:new01} (recall that we have assumed that $t>5/2$) as follows
\begin{eqnarray}
\bigg[\sum_{\ell\ne 0} 
 \| {\cal F}_{m+2\ell N} F\|_{W_{m}^j}
 \bigg]^2&\le& C_t N^{2j-2t} 
 \bigg[\sum_{\ell\ne 0} |m+2\ell N|^{2t-2j}
 \|{\cal F}_{m+2\ell N} F\|^2_{W_{m+2\ell N}^j} 
 \bigg]\nonumber\\
 &\le& C_t N^{2j-2t} \bigg[\sum_{\ell\ne 0} 
 \|{\cal F}_{m+2\ell N} F\|^2_{W_{m+2\ell N}^t} 
 \bigg]\label{eq:new02}
\end{eqnarray}
where in the last step we have applied inequality \eqref{eq:inequatityWms}. Plugging
\eqref{eq:new02} in \eqref{eq:new00}, and using \eqref{eq:Norm_decomposition}, we deduce finally 
\begin{eqnarray}
E_2\!\!&\le&\!\!  C(1+\log N)^s N^{2s-2t} \!\!\!\!\!\!
\sum_{-N+1\le m\le N}  \sum_{\ell\ne 0} \|{\cal F}_{m+2\ell N} F\|^2_{W_{m+2\ell N}^t} 
\le C(1+\log N)^s N^{2s-2t} \|F\|_{{\cal H}^t}^2.\qquad
  \label{eq:E2}
\end{eqnarray} 
}

Gathering bounds \modification{\eqref{eq:E3}, \eqref{eq:E1} and \eqref{eq:E2}},  we obtain the spectrally accurate
convergence estimate~\eqref{eq:main_res} in Theorem \ref{theo:main}.

\section{A FFT-based interpolatory cubature on the sphere}\label{sec:quad}

 As described in the introduction, interpolatory cubature rules on the sphere are
 important in several applications, including the radiative transfer and wave
 propagation models. Using the FFT-based spherical interpolatory operator,
 for a (wavenumber) parameter $\kappa$, 
 we develop a cubature rule to approximate the  following (non--, mild--, and highly--oscillatory) integral on the sphere:
 \begin{eqnarray}
  {\cal I}_{\kappa}F&:=&\int_{0}^\pi \! \int_{0}^{2\pi}\!F(\theta,\phi)\exp({\rm i}\kappa \cos\theta)
  \,\sin\theta\,{\rm d}\phi\,{\rm d}\theta=\iint_{\Sp} F^\circ({\bf x})\exp({\rm i}\kappa{\bf x}\cdot[0,0,1]) {\rm d}S({\bf x}).\label{eq:integral}
 \end{eqnarray}
 In the integral above the parameter $\kappa$ is a real  number. 
 Therefore, \eqref{eq:integral} includes
 standard integrals as well as a class of highly--oscillatory integrals for large
 values of $\kappa$. In wave propagation applications, the $[0, 0, 1]$ 
 corresponds to the direction of the incident wave. The rotationally invariant
 property of the sphere facilitates fixing such an incident direction. The above
 integral occurs, for example, in developing efficient computer models to 
 simulate scattered wave (and its far-field) from an acoustically/electromagnetically
 small, medium, and large closed obstacles~\cite{GH:09, GH:11, GH:13, GH:14} with
 compact  simply connected surface, leading to  surface integral reformulations of the model on the sphere.
 The integral for the $\kappa = 0$ case occurs in potential theory and radiative transport models. 
 
For the FFT-based efficient cubature approximation of the integral, we first consider a Filon-type
product integration interpolatory  approximation  
 \begin{eqnarray}
  {\cal I}_{N,\kappa}F&:=&\int_{0}^\pi \! \int_{0}^{2\pi}\! \big({\cal Q}_NF\big)(\theta,\phi)\exp({\rm i}\kappa \cos\theta)
 \sin\theta \, {\rm d}\phi\,{\rm d}\theta.
  \end{eqnarray}
  Using the representation  
  \[
   ({\cal Q}_NF)( \theta,\phi)=p_0(\theta)+\!\!\sum_{ -N/2  <m \le N/2\atop  \text{odd }m\ne0}\!\!\!\!
\sin\theta\, p_m(\theta)\exp({\rm i}m\phi)+\!\!\sum_{ -N/2  <m \le N/2\atop  \text{even }m\ne0}\!\!\!\!
\sin^2\theta\, p_m(\theta)\exp({\rm i}m\phi) 
  \]
  ($p_m\in\mathbb{D}_{N-2}^{\rm e}$ if $m\ne 0$), 
we obtain 
\[
 {\cal I}_{N,\kappa}F  =  {2\pi} \int_0^\pi p_0(\theta)\exp({\rm i}\kappa\cos\theta)  \sin\theta\,{\rm d\theta}=\sqrt{2\pi} \int_0^\pi \big({\cal F}_0 {\cal Q}_NF\big)(\theta)\exp({\rm i}\kappa\cos\theta)  \sin\theta\,{\rm d}\theta.
\]
From this property and Proposition \ref{prop:prop01} we  see how this cubature  rule
can efficiently implemented:
\begin{itemize}
 \item \label{item:fj0}Compute 
 \[
  f_{j,0}:=\frac{1}{2N}\sum_{k=0}^{2N-1} F(\theta_j,\phi_k)
 \]
 \item Construct
 \[
  \left(\alpha^{0}_\ell \right)_{\ell=0}^N :=  {\bf DCT}_N((f_{j,0})_{j=0,\ldots,N}), 
 \]
\item Return
\[
 {\cal I}_{N,\kappa}F=2\pi \sum_{\ell =0}^{N} \alpha_\ell\omega_\ell(\kappa)
\]
where
\begin{eqnarray}
 \omega_\ell(\kappa)&:=&2\pi \int_{0}^\pi \cos \ell\theta \exp({\rm i}\kappa\cos\theta)\sin\theta\,{\rm d}\theta=
 2\pi \int_{-1}^1 T_{\ell}(x)\exp({\rm i}\kappa x){\rm d}x.\label{eq:theweights}
\end{eqnarray}
\end{itemize}
The cost of computing $(\alpha_\ell)_{\ell=0}^N$ is about ${\cal O}(N^2)$, and is dominated by the 
first step of the algorithm. The weights \eqref{eq:theweights} ($T_\ell$ denotes the Chebyshev polynomial of degree $\ell$) can be computed 
in a stable and fast way in ${\cal O} (N)$ operations \cite{DoGrSm:2010}. 

For the error analysis  of the rule, based on~\eqref{eq:interp_new_rep}-\eqref{eq:01:proof:MainTheo}, we first arrive at the  following formula:
\begin{eqnarray}
 {\cal I}_{N,\kappa}F -{\cal I}_{\modification{\kappa}}F &=&\sqrt{2\pi} \int_0^\pi \Big[\big({\cal F}_0 {\cal Q}_NF\big)(\theta)-\big({\cal F}_0 F\big)(\theta)\Big] \sin\theta\,{\rm d}\theta\\
 &=&\sqrt{2\pi} \int_0^\pi \Big[\big({\rm q}_N^{\rm e}\rho_N^0F\big)(\theta)-\big({\cal F}_0 F\big)(\theta)\Big] \sin\theta\,{\rm d}\theta.
\end{eqnarray}
Thus, with  $e_N:={\rm q}_N^{\rm e}\rho_N^0F-{\cal F}_0F$, the cubature approximation error is bounded by  $\sqrt{2\pi}\|e_N\|_{L^2_{\rm sin}}$,  so 
that we first ensure the convergence should be independent of $\kappa$. Actually this estimate
can be improved  by performing integration by parts, to obtain high-order decay in the error
for large values of $\kappa$.  Using $e_N(0)=e_N(\pi)=0$, we obtain
\begin{eqnarray}
  {\cal I}_{N,\kappa}F -{\cal I}_{\modification{\kappa}}F &=& -\frac{\sqrt{2\pi}\, {\rm i}}{\kappa}\int_0^{\pi} e_N'(\theta)\exp({\rm i}\kappa\cos\theta)\,{\rm d}\theta
  \\ &=& 
  \frac{\sqrt{2\pi}}{\kappa^2}\bigg[
 \frac{1}{\sin\theta} e_N'(\theta) \exp({\rm i}\kappa\cos\theta)\bigg|_{\theta=0}^{\theta=\pi}
 \nonumber \\
  & & 
 -\int_0^{\pi}  \Big(\frac{1}{\sin\theta} e_N'(\theta)\Big)'\exp({\rm i}\kappa\cos\theta)\,{\rm d}\theta
  \bigg].
\end{eqnarray}
We observe that for sufficiently smooth $F$, $e_N'(0)=e_N'(\pi)=0$ and therefore the pointwise value of 
$\frac{1}{\sin (\cdot) } e_N'(\cdot)$ at these points as well as the last integral are well defined. 

Below we  present the error estimate and convergence result for the cubature rule.
We omit a detailed  analysis of the estimate since 
it can be proved 
using  arguments similar to that  we developed (for a similar rule)  and analyzed in~\cite[Section 5]{DoGa:2012}.

\begin{theorem}\label{theo:convQuadRule}
Let $F \in {{\cal H}^r}$. For $\ell=0,1$ and $r>3/2$ or $\ell=2$ and
$r>4$  
\begin{equation}\label{eq:02:theo:convQuadRule}
  | I_\kappa (F)- I_{\kappa,N}  (F)|   \le C_r \kappa^{-\ell}
N^{\eta(\ell)-r}\|F\|_{{\cal H}^r},
\end{equation}
 where  $C_r$  independent of $N$ and  
\[
 \eta(\ell):=\left\{
\begin{array}{ll}
0,\quad &\text{if $\ell=0$},\\
 3/2,\quad&\text{if $\ell=1$},\\
4,\quad&\text{if $\ell=2$}.
\end{array}
\right.
\]
\end{theorem}
\section{Numerical experiments}\label{sec:num_exp}

In this section we demonstrate the main interpolatory spectrally accurate approximation
result~\eqref{eq:main_res} and  the high-order  cubature approximation result~\eqref{eq:02:theo:convQuadRule}
for functions with various order of smoothness. We  also demonstrate that the \modificationG{construction} of full FFT-based interpolatory
approximation developed in this article using  the uniform-grid and the ${\cal Q}_N$ operator  is faster, 
even for small to medium sized data locations, than another efficient 
similarly accurate interpolation operator ${\cal Q}_N^{\rm gl}$.  We developed 
the operator ${\cal Q}_N^{\rm gl}$ in~\cite{DoGa:2012}, using Gauss-Lobatto points in latitudinal angle, 
that facilitates the use of the standard FFT only in the azimuthal variable. 

For calculation of the ${\cal H}^t$ norms, for $t = 0, 1$, used
in~\eqref{eq:main_res}, 
we  apply the following integral based formulas:
\begin{eqnarray} 
\hspace{-0.3in} \|F\|_{{\cal H}^0}^2&:=& \int_0^\pi\!\!
\int_0^{2\pi}|F(\theta,\phi)|^2\sin\theta\,{\rm d}\phi\,{\rm d}\theta \label{eq:H0norm} \\
\hspace{-0.3in} \|F\|_{{\cal H}^1}^2&:=& \frac{1}{4} \|F\|_{{\cal H}^0}^2 \int_{0}^\pi\!\int_{0}^{2\pi}
\bigg|\frac{\partial
F}{\partial \phi}(\theta,\phi)\bigg|^2\frac{1}{\sin\theta}\,{\rm d}\phi{\rm
d\theta}+\int_{0}^\pi\!\int_{0}^{2\pi}
\bigg|\frac{\partial
F}{\partial \theta}(\theta,\phi)\bigg|^2\ {\sin\theta}\,{\rm
d}\phi\,{\rm d\theta}.    \label{eq:H1norm}
\end{eqnarray}
Except for some trivial cases, the above norms cannot be evaluated exactly. We computed the above
norms for tabulated results in this section using over $150,000$ quadrature points on the sphere,
taking into account that some of the  functions considered in this section have only limited smoothness properties and hence require fine grids  to compute with sufficiently high
accuracy.

\subsection*{Experiment \#1 (Approximation of smooth and limited smooth functions)}

For the  first set of  experiments we consider  interpolatory approximation of test functions:
\[
 F_1^\circ(x,y,z):=\frac{1}{4+x+y+z},\qquad  
 F_j^\circ (x,y,z):=(1-x^2)^{5/2-j}yz, \quad \text{for $j=2,3$} \quad \text{and}\quad  (x,y,z) \in \Sp.
\]
Recalling~\eqref{eq:x}-\eqref{eq:01}, the corresponding equivalent functions are
\begin{eqnarray*}
  F_1(\theta,\phi)&:=&\frac{1}{4+\sin\theta\cos\phi+ \sin\theta\sin\phi  +\cos\theta} \\
  F_j(\theta,\phi)&:=&(1-\sin^2\theta \cos^2\phi)^{5/2-j}\sin\theta\sin\phi\cos\theta,\quad
  \text{$j=2,3$}. 
\end{eqnarray*}
Clearly $F^\circ_1$ is a smooth function, and hence our theoretical result~\eqref{eq:main_res} 
suggests superalgebraic convergence  ${\cal Q}_NF^\circ_1$  to $F^\circ_1$   in both the 
${\cal H}^0$ and ${\cal H}^1$ norms. Computational results in Table~\ref{Tab:01} validate the theoretical result
and demonstrate the power of obtaining machine precision accurate approximation of the smooth
function with $N = 32$. 

\begin{table}[ht]
\centerline{
\begin{tabular}{ |c|c|c|c|c| }
\hline
$N$& $\|{\cal Q}_NF_1-F_1\|_{{\cal H}^0}$ &EoC &$\|{\cal
Q}_NF_1-F_1\|_{{\cal H}^1}$&EoC\\
\hline
008& 4.86{\rm E}-06 & & 4.40{\rm E}-05&\\
016& 2.02{\rm E}-11 & 17.9& 3.37{\rm E}-10 &17.0\\
032& 5.78{\rm E}-15 & 11.8&  7.82{\rm E}-15 &15.4\\
\hline
\end{tabular}}
\caption{Approximation  of $F_1$ by ${\cal Q}_NF_1$: Errors and estimate order of convergence  (EoC) 
\label{Tab:01}}
 \end{table} 

The functions $F_2^\circ$ and $F_3^\circ$ have only limited regularity.
It can be shown that, for any $\varepsilon > 0$,  $F_2^\circ\in{\cal H}^{4-\varepsilon}$ and
$F_3\in {\cal H}^{2-\varepsilon}$. Indeed, using an atlas with local charts around 
the singularities [points $(\pm 1,0,0) \in \Sp$] one can easily see that the Sobolev regularity of $F_j^\circ$ 
are the same as the functions $\tilde{F}^\circ_j(y,z):= (y^2+z^2)^{5/2-j}y$. 
Hence according to our theoretical result~\eqref{eq:main_res}, the estimated order
of convergence (EoC) in approximating $F_2$ by  ${\cal Q}_NF_2$  in  the ${\cal H}^0, {\cal H}^1$
norms are respectively {\em almost}  $4$ and $3$ and
that for  $F_3$ by ${\cal Q}_NF_3$  are respectively {\em almost} $2$ and $1$. Computational results in
Table~\ref{Tab:02}  validate the theoretical result~\eqref{eq:main_res}.

\begin{table}[h]
\centerline{
\begin{tabular}{ |c|c|c|c|c| }
\hline
$N$& $\|{\cal Q}_NF_2-F_2\|_{{\cal H}^0}$ &EoC &$\|{\cal
Q}_NF_2-F_2\|_{{\cal H}^1}$&EoC \\
\hline 
008&  1.43{\rm E}-03  &             &  1.41{\rm E}-02  &           \\ 
016&  8.14{\rm E}-05 &  4.14 &  1.55{\rm E}-03  & 3.18\\
032&  5.01{\rm E}-06 &  4.02 &  1.90{\rm E}-04  & 3.03\\
064&  3.12{\rm E}-07 &  4.01 &  2.36{\rm E}-05  & 3.00\\
128&  1.87{\rm E}-08 &  4.06 &  3.07{\rm E}-06  & 2.95\\
\hline
\end{tabular}}

\ \\

\centerline{
\begin{tabular}{ |c|c|c|c|c| }
\hline
$N$& $\|{\cal Q}_NF_3-F_3\|_{{\cal H}^0}$ &EoC &$\|{\cal
Q}_NF_3-F_3\|_{{\cal H}^1}$&EoC \\
\hline
008&  2.76{\rm E}-02 &     &    3.21{\rm E}-01  &     \\
016& 6.92{\rm E}-03 &  2.00 &  1.58{\rm E}-01  & 1.02\\
032& 1.73{\rm E}-03 &   2.00 &  7.93{\rm E}-02 &  1.00\\
064&4.33{\rm E}-04 &  2.00 &  3.97{\rm E}-02  & 1.00\\
128&1.08{\rm E}-04 &  2.00 &  1.99{\rm E}-02 &  0.99 \\
\hline
\end{tabular}}
\caption{Approximation of $F_2$ and $F_3$ by ${\cal Q}_NF_2$ and  ${\cal Q}_NF_3$: Errors and EoC
\label{Tab:02}}
\end{table}

\subsection*{Experiment \#2 (Accuracy and fast evaluation comparison with a recent  work)}

For this experiment we compare the performance \modificationG{in construction}, in terms of error and computation time, of the FFT-based interpolant  developed in this article  with the interpolant considered 
in~\cite{DoGa:2012} (and first proposed, not analyzed, in \cite{GaMha:2006}). This interpolant
shares the same discrete space, $\chi_N$, and the nodes in the azimuthal angle $\{\phi_j\}$. The difference 
is  on the nodes in the latitudinal angle which were chosen in~\cite{DoGa:2012} 
to be the non-uniform grid points $\{\theta_i = \arccos{\eta}_i\}_{i=0}^N$ where
$\{\eta_i\}_{i=0}^N$ are the Gauss-Lobatto points of the quadrature rule for approximating integrals in $[-1,1]$. In other 
words, $\eta_0=-1$, $\eta_N=1$ and $\eta_i$ for $i=1,\ldots,N-1$ are the roots of $P_N'(x)$ where
$P_N$ is the Lagrange polynomial of degree $N$. We recall that this 
non-uniform Gauss-Lobatto points based interpolant is denoted as  ${\cal Q}_N^{\rm gl}$.

In~\cite{DoGa:2012} we proved that for $F \in  {\cal H}^t$ and $s = 0, 1$, 
\[
 \|{\cal Q}^{\rm gl}_N F-F\|_{{\cal H}^s}\le CN^{s-t}
 \|F\|_{{\cal H}^t},
\]
which is, up to the $(\log N)^{s/2}$ term,  identical to the error estimate in~\eqref{eq:main_res} that we proved for the FFT-based operator  ${\cal Q}_N$. 
For our comparison testing purpose, we have chosen the function
\[
 F_4(\theta,\phi)=\left[\left(\frac{1}{\sqrt{3}}-\sin\theta\cos\phi \right)^{2}+
 \left(\frac{1}{\sqrt{3}}-\sin\theta\sin\phi \right)^{2}+
 \left(\frac{1}{\sqrt{3}}-\cos\theta  \right)^{2}\right]^{3/2}\in{\cal H}^{4-\varepsilon}
\]
which   corresponds to the  function $F_4^\circ(\xh)=\left|\xh - \x^*\right|^3,~  \xh\in\Sp$ with
 $\x^* = [1/\sqrt{3},1/\sqrt{3},1/\sqrt{3}]$.
The error in ${\cal H}^0$ and ${\cal H}^1$ norms and associated EoC  are depicted in Table \ref{Tab:04} for  ${\cal Q}_N$ 
and ${\cal Q}_N^{\rm gl}$. Similar to our established
theoretical results, we  observe from Table~\ref{Tab:04} 
that although ${\cal Q}^{\rm gl}_N$ performs slightly better, the difference is not significant
and the estimated orders of convergence are roughly the same.

\begin{table}[h]
\centerline{
\begin{tabular}{ |c|c|c|c|c|c| }
\hline
$N$& $\|{\cal Q}_NF_4-F_4\|_{{\cal H}^0}$ &EoC&$\|{\cal
Q}_N^{\rm gl}F_4-F_4\|_{{\cal H}^0}$&EoC\\
\hline 
008& 1.48e-03 &       & 1.41e-03 &       \\ 
016& 1.00e-04 &  3.89 & 9.03e-05 &  3.97 \\ 
032& 6.16e-06 &  4.02 & 5.68e-06 &  3.99 \\ 
064& 3.66e-07 &  4.07 & 3.61e-07 &  3.98 \\ 
128& 2.63e-08 &  3.80 & 2.47e-08 &  3.87   \\
\hline
\end{tabular}}

\ \\[1ex]

\centerline{
\begin{tabular}{ |c|c|c|c|c| }
\hline
$N$& $\|{\cal Q}_NF_4-F_4\|_{{\cal H}^1}$ &EoC &$\|{\cal
Q}_N^{\rm gl}F_4-F_4\|_{{\cal H}^1}$&EoC\\
\hline 
008& 1.43e-02 &       & 1.50e-02 &       \\  
016& 1.88e-03 &  2.93 & 1.90e-03 &  2.99 \\   
032& 2.27e-04 &  3.05 & 2.20e-04 &  3.11 \\  
064& 2.69e-05 &  3.08 & 2.83e-05 &  2.96 \\  
128& 3.82e-06 &  2.82 & 4.26e-06 &  2.73  \\
\hline  
\end{tabular}}

\caption{Approximation of $F_4$ by ${\cal Q}_NF_4$  and 
${\cal Q}^{\rm gl}F_4$:  ${\cal H}^0$  case (top) and
 ${\cal H}^1$ case (bottom). \label{Tab:04}}
\end{table} 

\modificationG{It is important to observe the difference
between the {\em construction} and {\em evaluation} of the interpolation operators in the current article and that in~\cite{DoGa:2012}. After construction
of these two interpolation operators (see Figure~\ref{fig:low-high} for construction CPU time), computing   our two  interpolatory approximations
at various observation points on $\Sp$ requires same basis function evaluations at the points, as they share the same approximation space $\chi_N$. Unlike    standard spherical harmonics based polynomial approximations,  construction and evaluation of our interpolation operators
do not involve Legendre polynomials. Hence  we do not require use of techniques such as fast Legendre transforms for our non-polynomial approximations.
Our basis functions are trigonometric polynomials and hence for evaluation of both the interpolation operators,  standard FFT or NFFT~\cite{nfft}
techniques can be used, depending on whether the observation points are equally
spaced or not. }

The computational  effort required for \modificationG{construction of both the interpolants}  are however important.
In Figure~\ref{fig:low-high}, we show 
the low computational cost of the spherical interpolant constructed in this article 
compared to even the efficient matrix-free interpolant developed in~\cite{DoGa:2012}.
The difference in the performance and computational complexity \modificationG{for construction} of the interpolants  
can be easily explained just by examining both interpolants.   For construction of the  ${\cal Q}_{N} $ based
approximation, we proceed  as follows (see Proposition \ref{prop:01}): first, we apply $N-1$ 
FFT transforms of $2N$ elements, corresponding to odd node indices and next, for even cases, apply  $N+1$ DCT/DST (discrete sine/cosine transform) . 
Thus the overall computational complexity for the ${\cal Q}_{N}$ operator 
approximation  is 
${\cal O}({N}^2\log N)$ operations. For construction of the
 ${\cal Q}_{N}^{\rm gl}$ based approximation, the first step is similar, with  $N-1$ FFT
transforms of $2N$ elements, but the second step (in the non-uniform grid latitudinal angle) is different:  two different (polynomial) interpolation problems of $N+1$ (${\rm{p}}_N^{2m}$) and $N-1$ (${\rm{p}}_N^{2m-1}$) have to be solved  $N$ times which amounts about ${\cal O}(N^3)$ 
operations.  Moreover, the  interpolant ${\cal Q}_N$ developed in this article can 
be used in a natural way with nested grids which can be easily exploited to construct  error estimates almost for free, or, if the data node points are doubled,  previous function evaluations can be reused to construct the associated updated interpolatory approximation.

\begin{center}
\begin{figure}[h]
\includegraphics[width=0.9\textwidth]{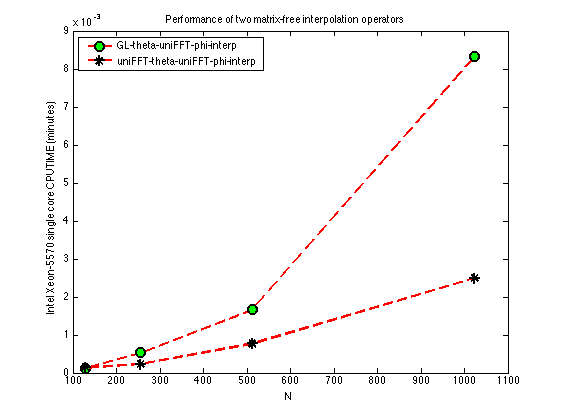} 

\vspace{-0.2in}
\includegraphics[width=0.9\textwidth]{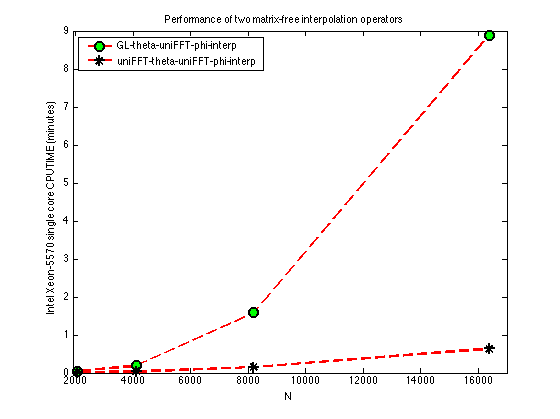} 
\vspace{-0.4in}
\caption{\label{fig:low-high} Performance of the \modificationG{construction} of  operator ${\cal Q}_N$ (*-values) compared
to that of ${\cal Q}_N^{\rm gl}$ (o--values) for various data parameter $N = 2^j$, 
[top: $j = 7, 8, 9, 10$ and bottom:  $j = 11, 12, 13, 14$] applied to the
data obtained using the function $F_4$. The ordinate values are the CPUTIME (in minutes) obtained
using a single core of a Intel Xeon-5570 2.93GHz processor.  The ${\cal Q}_N$ operator facilitates application
of  the standard FFT
in both the azimuthal and latitudinal angles while the  ${\cal Q}_N^{\rm gl}$  with non-uniform 
Gauss-Lobatto based points in the latitudinal angles provides the standard FFT evaluation only
in the azimuthal variable.
}
\end{figure}
\end{center}
 
\clearpage

\subsection*{Experiment \#3 (Performance of the FFT-based cubature)}

Using the smooth integrand  function $F_2$ and for the limited smooth integrand 
$F_4$, we have tested the convergence of the cubature rule developed in Section~\ref{sec:quad} 
to approximate the integral~\eqref{eq:integral}, for $\kappa=0,10,10^2,\ldots,10^6$,
corresponding to the total  
non-, mildly-, and highly-oscillatory integrands. The results  for all these
cases are given in Tables \ref{Tab:06} and
\ref{Tab:07}. 

Each row of Tables \ref{Tab:06} and \ref{Tab:07}  corresponds to a fixed $N$.
We clearly observe the $\kappa^{-2}$ decay of the error, for fixed $N$. Reading along
 the columns corresponds to varying values of $N$. For  the smooth 
 function $F_1$,  based on the column results in Table \ref{Tab:06}, 
 we see the superalgebraic convergence in $N$, as proved in Section~1.
 For a fixed $\kappa$, for the case of the integrand function $F_4$ with
 limited regularity,  we observe from Table~\ref{Tab:07} the convergence 
 in general is better than 
 the estimated theoretical result, suggesting that our estimated error result and
 analysis  could be improved in some cases.

\begin{table}[ht]
 \begin{tabular}{ |c|c|c|c|c|c|c|c|c|}
 \hline 
 $N\setminus\kappa$ & 0         & 1         & 10      & 100       & $10^3$      & $10^4$      & $10^5 $     & $10^6$\\     
 \hline
 004 & 5.36{\rm E-}05  & 1.03{\rm E-}04  & 5.01{\rm E-}05  & 6.77{\rm E-}07  & 4.55{\rm E-}09  & 6.93{\rm E-}11  & 7.22{\rm E-}13  & 6.83{\rm E-}15\\
 008 & 1.17{\rm E-}08  & 1.28{\rm E-}08  & 9.14{\rm E-}07  & 4.28{\rm E-}09  & 5.04{\rm E-}11  & 4.35{\rm E-}13  & 4.24{\rm E-}15  & 4.38{\rm E-}17\\
 016 & 4.44{\rm E-}16  & 1.28{\rm E-}14  & 2.15{\rm E-}13  & 9.89{\rm E-}14  & 7.01{\rm E-}16  & 2.23{\rm E-}18  & 1.79{\rm E-}19  & 6.07{\rm E-}20\\  
\hline
 \end{tabular}
\caption{$|{\cal I}_\kappa(F_1)-{\cal I}_{\kappa,N}(F_1)|$ for various parameters  $\kappa$ and $N$\label{Tab:06}}
\end{table}

\begin{table}[ht]
 \begin{tabular}{ |c|c|c|c|c|c|c|c|c| }
 \hline
$ N\setminus\kappa$ & 0         & 1         & 10      & 100       & $10^3$      & $10^4$      & $10^5 $     & $10^6$\\     
 \hline 
 004 & 1.04{\rm E-}04  & 2.67{\rm E-}03  & 1.72{\rm E-}03  & 2.40{\rm E-}05  & 1.64{\rm E-}07  & 2.48{\rm E-}09  & 2.58{\rm E-}11  & 2.44{\rm E-}13  \\ 
 008 & 7.92{\rm E-}05  & 8.62{\rm E-}05  & 5.75{\rm E-}04  & 3.15{\rm E-}06  & 2.95{\rm E-}08  & 3.08{\rm E-}10  & 3.10{\rm E-}12  & 3.08{\rm E-}14  \\ 
 016 & 4.14{\rm E-}06  & 4.20{\rm E-}06  & 9.47{\rm E-}06  & 1.31{\rm E-}07  & 1.02{\rm E-}09  & 2.01{\rm E-}11  & 2.11{\rm E-}13  & 1.99{\rm E-}15  \\ 
 032 & 2.22{\rm E-}08  & 2.33{\rm E-}08  & 8.44{\rm E-}08  & 2.26{\rm E-}08  & 3.67{\rm E-}11  & 1.00{\rm E-}12  & 1.07{\rm E-}14  & 1.01{\rm E-}16  \\ 
 064 & 2.51{\rm E-}09  & 2.51{\rm E-}09  & 2.73{\rm E-}09  & 1.83{\rm E-}08  & 4.98{\rm E-}12  & 3.78{\rm E-}14  & 4.02{\rm E-}16  & 3.93{\rm E-}18  \\ 
 128 & 1.54{\rm E-}10  & 1.54{\rm E-}10  & 1.56{\rm E-}10  & 4.20{\rm E-}10  & 2.21{\rm E-}12  & 5.97{\rm E-}15  & 1.03{\rm E-}16  & 9.62{\rm E-}19  \\
 256 & 3.33{\rm E-}12  & 3.31{\rm E-}12  & 3.34{\rm E-}12  & 4.93{\rm E-}12  & 4.26{\rm E-}13  & 7.06{\rm E-}16  & 1.11{\rm E-}17  & 2.58{\rm E-}19  \\ 
 \hline
 \end{tabular}
\caption{$ |{\cal I}_\kappa(F_4)-{\cal I}_{\kappa,N}(F_4)| $  for various parameters  $\kappa$ and $N$\label{Tab:07}}
\end{table}

\appendix 
\section{Discussion on  and verification of Hypothesis \ref{hypo:01}}\label{app:hyp_app}

In this section, we present details required to computationally verify the 
inequalities~\eqref{eq:H:00} in 
Hypothesis~\ref{hypo:01} and demonstrate that the hypothesis holds for
almost all practical values of $N$. To this end, we first
rewrite Hypothesis \ref{hypo:01} in a computationally 
convenient form. Using the cosine change of variables $x=\cos\theta$ 
and the Chebyshev polynomial $T_{2N}$,  we rewrite inequality 
~\eqref{eq:H:00} as
\begin{equation}
 \label{eq:A:01} 
 -2 \int_{-1}^1 
| p_{N-2}(x)|^2
(1-x^2)^\alpha T_{2N}(x)\,{\rm d}x\le c_{\rm H}^{\alpha} 
\int_{-1}^1 |p_{N-2}(x)|^2 (1-x^2)^{\alpha}\,{\rm d}x,\quad \forall p_{N-2}\in\mathbb{P}_{N-2},
\end{equation}
 for $\alpha=0,1,2$ where  $c_{\rm H}^{\alpha}  <1$.

Next we consider a polynomial of degree $N-2$ represented using the Chebyshev basis:
\[
 p_{N-2}=\sum_{j=0}^{N-2}\beta_{j+1} T_j.
\]
Then \eqref{eq:A:01}  is equivalent to the coefficient based inequality 
\begin{equation}
 -2\sum_{i,j=1}^{N-1}\beta_i\beta_j b_{ij}(\alpha,N)\le c_{\rm H}^{\alpha}
 \sum_{i,j=1}^{N-1}\beta_i\beta_j a_{ij}(\alpha)
\end{equation}
where
\begin{eqnarray*}
a_{ij}(\alpha)&:=&\int_{-1}^1 T_{i-1}(x) T_{j-1}(x)(1-x^2)^\alpha\,{\rm d}x,\quad \\
b_{ij}(\alpha,N)&:=& \int_{-1}^1 T_{i-1}(x) T_{j-1}(x) T_{2N}(x) (1-x^2)^\alpha\,{\rm d}x. 
\end{eqnarray*}
Observe that these quantities are easily computable using the identities
\begin{subequations}
 \label{eq:A:02} 
\begin{eqnarray}
 T_i T_j&=&\frac{1}2 \big(T_{i+j}+T_{|i-j|}\big)\label{eq:a:A:02} \\
 T_{2N}&=&2T_{N}^2-1\label{eq:b:A:02} \\
 (1-x^2)&=& -\frac{1}2T_2(x)+\frac12,\quad (1-x^2)^2=\frac{1}4T_2^2(x)-\frac{1}2
T_2(x)+\frac14\label{eq:c:A:02} \\
 \int_{-1}^1 T_i(x) \,{\rm d}x&=&\left\{\begin{array}{ll}
                                  -\frac{2}{i^2-1},\ &\text{for even $i$,}      
\\
                                  0,&\text{otherwise.}                         
             
                                         \end{array}
\right. \label{eq:d:A:02}
\end{eqnarray}
\end{subequations}
Thus, for numerical verification of the  Hypothesis~\ref{hypo:01}, we used 
the follow algorithm:
\begin{itemize}
 \item Construct an auxiliary matrix $C$ sufficiently large  with $C_{i+1,j+1}=\int_{-1}^1 T_i T_j$ using
 \eqref{eq:a:A:02} and  \eqref{eq:d:A:02} . 
 \item From $C$, construct $A(\alpha,N):=(a_{ij}(\alpha))_{i,j=1}^{N-1}$ 
 and $B(\alpha,N):=(b_{ij}(\alpha,N))_{i,j=1}^{N-1}$ by  applying
 \eqref{eq:a:A:02}--\eqref{eq:d:A:02}. Hence
 \begin{eqnarray*}
  a_{ij}(0)&=&c_{ij}\\
  a_{ij}(1)&=& -\frac{1}{4} \big[c_{i+2,j}+c_{|i-2|+1,j}\big]+\frac{1}2 c_{ij}\\
  a_{ij}(2)&=&  \frac{1}{16} \big[c_{i+2,j+2}+c_{|i-2|+1,j+2}+c_{i+2,|j-2|+1}+c_{|i-2|+1,|j-2|+1}\big]\\
  &&- 
  \frac{1}4 \big[c_{i+2,j}+c_{|i-2|+1,j}\big]+\frac14c_{ij}\\
    b_{ij}(\alpha,N)&=&  \frac{1}{2}
\big[a_{i+N,j+N}(\alpha)+a_{|i-N|+1,j}(\alpha)+a_{i,|j-N|+1}(\alpha)
    +a_{|i-N|+1,|j-N|+1}(\alpha)\big]
    -  a_{i,j}(\alpha).\\  
 \end{eqnarray*}

 \item Compute the minimum of the generalized Rayleigh quotient for $A$ and $B$ 
 \begin{equation}
 \label{eq:A:03} 
 c_{\rm H}(N;\alpha):=
  -2\min_{\mathbf{b}\in\mathbb{R}^{N-1}} 
  \frac{\mathbf{b}^\top B(\alpha,N)\mathbf{b}}{\mathbf{b}^\top A(\alpha,N) \mathbf{b}}
 \end{equation}

\end{itemize}
Both $A(\alpha,N)$ and $B(\alpha,N)$ are symmetric. Moreover,  $A(\alpha,N)$ is positive definite. Thus, we can compute 
the Cholesky factorization $A=R^\top R$ so that  \eqref{eq:A:03} is equivalent to compute the 
smallest algebraic eigenvalue
of the matrix $R^{-\top} B(\alpha,N)R^{-1}$.

We have implemented the above algorithm by  precomputing $C$ that facilities
acceleration of the algorithm for several values of $N$.  The graphs of 
$c_{\rm H}(\cdot;\alpha)$ for  $2 \leq N \leq 2^{14}$ are depicted in Figure \ref{fig:hyp-low-high}
demonstrating the validity of the hypothesis for most practically useful values of $N$.


\begin{center}
\begin{figure}[h]
\includegraphics[width=0.9\textwidth]{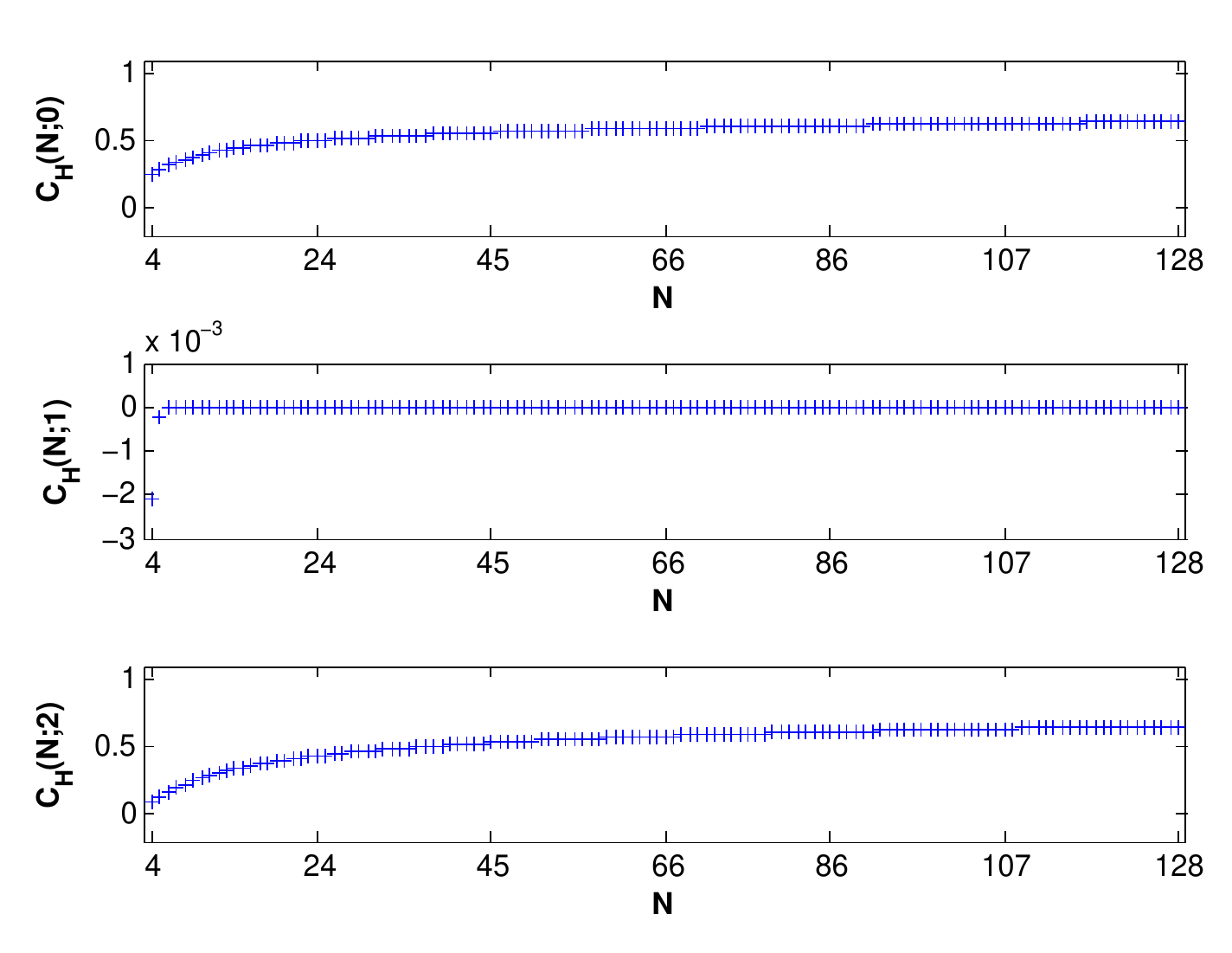} 

\vspace{-0.2in}
\includegraphics[width=0.9\textwidth]{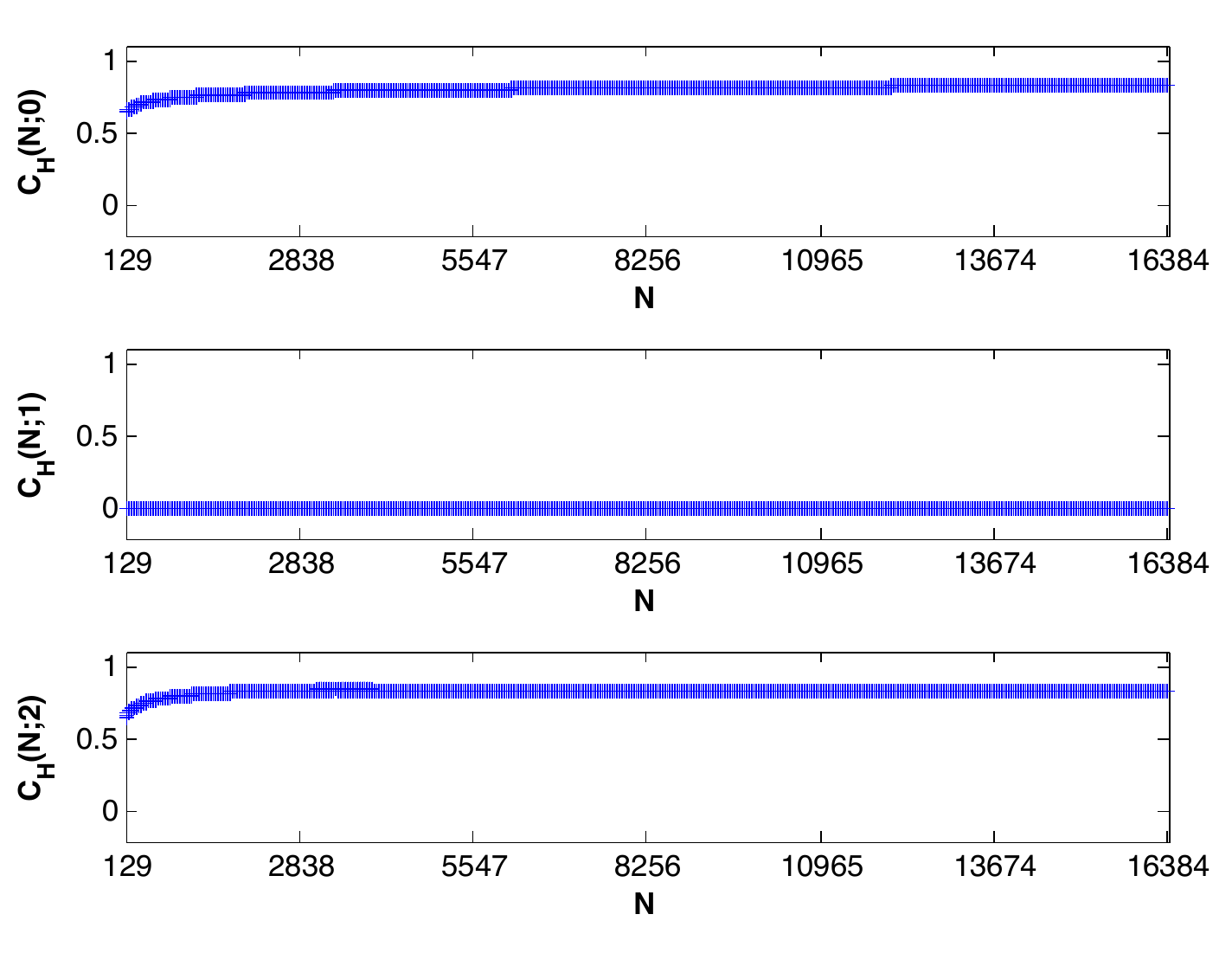} 
\vspace{-0.4in}
\caption{\label{fig:hyp-low-high} 
Plots of computed values of  $c_{\rm H}( N;\alpha)$ defined in~\eqref{eq:A:03} for 
$\alpha=0, 1, 2$, respectively in the first three subplots, for $2 \leq N \leq 2^7$;  and in the 
last three subplots for  $2^7  <  N \leq 2^{14}$. These plots, demonstrating 
$c_{\rm H}( N;\alpha) < 1$, numerically validate (for almost all practical cases)  
that in~\eqref {eq:H:00}, we have $c_{\rm H}^{(j)} <1$, for $j = 1, 2, 3$.
}
\end{figure}
\end{center}
 
\clearpage 

\section{Proofs results in Section~\ref{sec:3}}\label{sec:app_proof_sec_3}
In this section we provide proofs of 
Theorem~\ref{theo:equiv_norms} and Proposition~\ref{prop:inclusion_in_Hr}.

\subsection{Proof of Theorem~\ref{theo:equiv_norms}}
\label{sec:theorem3.1}
\begin{proof} 
First we recall that the space
\[
 {\cal H}^1=\big\{F:\R^2\to\mathbb{C}\ : \ F\in {\cal H}^{0}, \nabla_{\Sp}
F\in
{\cal H}^{0}\times {\cal H}^{0} \big\}
\]
is equipped with the norm
\begin{eqnarray}
 \|F\|_{{\cal H}^1}^2&:=&\frac{1}{4}\|F\|_{{\cal
H}^0}^2+\|\nabla_{\Sp}F\|_{{\cal  H}^0\times {\cal  H}^0}^2\nonumber\\
&=&\frac{1}{4}\int_{0}^\pi\!\int_{0}^{2\pi} |F(\theta,
\phi)|^2\sin\theta\,{\rm d}\theta\nonumber\\
&&+\int_{0}^\pi\!\int_{0}^{2\pi}
\bigg|\frac{\partial
F}{\partial \phi}(\theta,\phi)\bigg|^2\frac{1}{\sin\theta}{\rm
d\theta}\,{\rm d}\phi+\int_{0}^\pi\!\int_{0}^{2\pi}
\bigg|\frac{\partial
F}{\partial \theta}(\theta,\phi)\bigg|^2\ {\sin\theta}\,{\rm d\theta}\,{\rm
d}\phi.  \label{eq:normH1}
\end{eqnarray}
Since, 
\[
|\nabla_{\Sp}
(f\otimes e_m)(\theta,\phi)|^2 =\frac{1}{2\pi}\Big[ \frac{m^2}{\sin^2\theta}
 |f(\theta)|^2  +|f'(\theta)|^2\Big],
\]
we obtain for all $m\in\Z$,
\begin{eqnarray}
  \|f\|_{W^1_m}^2&=& \frac{1}{4}\int_0^{\pi}
|f(\theta)|^2\sin\theta\,{\rm d}\theta+  m^2\int_0^\pi |f(\theta)|^2\frac{{\rm
d}\theta}{\sin\theta} +\int_0^\pi
|f'(\theta)|^2\,\sin\theta\,{\rm d}\theta\label{eq:01:theo:eqWm1Y1}.
\end{eqnarray}
For $|m|\ge 1$, \eqref{eq:01:cor:equivNorms} follows from
\eqref{eq:01:theo:eqWm1Y1}.

Next we prove that if $|m|\ge 2$,
\begin{eqnarray}
  \|f\|_{W^2_m}^2&=&\frac{1}{16}\int_0^\pi
|f(\theta)|^2\sin\theta\,{\rm d}\theta+
\frac{9m^2}{4}\int_0^\pi
|f(\theta)|^2\frac{{\rm d}\theta}{\sin\theta}+
(m^4-4m^2)\int_0^\pi
|f(\theta)|^2\frac{{\rm d}\theta}{\sin^3\theta}\nonumber\\
&&+\frac{1}{4}\int_0^\pi |f'(\theta)|^2\sin\theta\,{\rm
d}\theta+
(1+2m^2)\int_0^\pi |f'(\theta)|^2\frac{{\rm d}\theta}{\sin\theta}+\int_0^\pi
|f''(\theta)|^2\,{\sin \theta}\,{\rm d}\theta. \nonumber \\
\label{eq:02:theo:eqWm1Y1}
\end{eqnarray}
It is convenient to recall that
\[
\|F\|_{{\cal H}^2}^2:=\|\Delta_{\Sp} F\|_{{\cal
H}^0}^2+\frac{1}{2}\|\nabla_{\Sp}
F\|_{{\cal H}^0\times {\cal
H}^0}^2+\frac{1}{16}\|F\|_{{\cal H}^0}^2,\label{eq:normH2}
\]
\modificationG{ Where $\Delta_{\Sp}$ is the Laplace-Beltrami operator on the sphere~\cite{Ne:2001}.}

Without loss of
generality, we  assume $f$ to be a real valued function.
The proof of \eqref{eq:02:theo:eqWm1Y1} requires more calculations and 
application of integration by parts several times to take care of some
cross products appearing in the integral form of the norm. Without
loss of generality,  for a fixed $m \in \Z$, we can assume $f\in {\rm
span}\left\{ Q_n^m\ :  n\ge |m|\right\}$ because this subspace
is dense in $W_m^2$.
Observe that
\begin{equation}\label{eq:f:cancellation}
f(0)=f'(0)=f(\pi)=f'(\pi)=0.
\end{equation}
Since  
\begin{eqnarray*}
\frac{1}{4}\|\nabla_{\Sp} f
\otimes e_{m}\|_{{\cal H}^0\times {\cal
H}^0}^2+\frac{1}{16}\|f\otimes e_{m}\|_{{\cal H}^0}^2&=&
\frac{m^2}{4}\int_0^\pi
f^2(\theta)\frac{{\rm d}\theta}{\sin\theta}+\frac{1}{4}\int_0^\pi
|f'(\theta)|^2{\sin\theta}  \,{\rm d}\theta\\
&&+
 \frac{1}{16}\int_0^\pi
f^2(\theta)\sin\theta\,{\rm d}\theta, 
\end{eqnarray*} it is sufficient  to analyze the  term containing the
Laplace-Beltrami operator:
\begin{eqnarray}
 \|\Delta_{\Sp} (f\otimes e_{m})\|_{{\cal H}^0}^2&=& \int_0^{\pi}
 \int_0^{2\pi}  |\Delta_{\Sp} (f\otimes e_m)(\theta,\phi)|^2\sin\theta\,{\rm
d}\theta\,{\rm d}\phi\nonumber\\
&=& 
\int_0^{\pi}
   \bigg|-\frac{m^2}{\sin^2\theta} f(\theta)+
\frac{1}{\sin\theta}\big|\big(\sin\theta
f'(\theta)\big)'\bigg|^2 \sin\theta\,{\rm
d} \theta\nonumber\\
&=&
 {m^4}\int_0^\pi f^2(\theta)\,\frac{{\rm d}\theta}{\sin^3\theta}+\int_0^\pi
 \frac{1}{\sin\theta}\big|\big(\sin\theta
f'(\theta)\big)'\big|^2\,{\rm d}\theta\nonumber\\
&& -
2m^2 \int_0^\pi f(\theta) \big(\sin\theta
f'(\theta)\big)'  \frac{{\rm
 d}\theta}{\sin^2\theta}  =:m^4 I_1+I_2-2m^2I_3.\
\quad\label{eq:03:theo:eqWm1Y1}
\end{eqnarray} 
Using~\eqref{eq:f:cancellation} and integration by parts, cubature
\begin{eqnarray}
 I_2&=&\int_0^\pi  \Big [ \frac{\cos^2\theta}{\sin\theta} |f'(\theta)|^2\
+\sin\theta|f''(\theta)|^2+\cos\theta
\big(|f'(\theta)|^2\big)'\Big]  \,{\rm d}\theta \nonumber\\
&=&
\int_0^\pi \frac{\cos^2\theta}{\sin\theta} |f'(\theta)|^2\,{\rm
d}\theta+\int_0^\pi \sin\theta|f''(\theta)|^2\,{\rm
d}\theta+\int_0^\pi
 |f'(\theta)|^2\sin\theta \,{\rm d}\theta\nonumber
\\&=&
\int_0^\pi  |f'(\theta)|^2\,\frac{{\rm
d}\theta}{\sin\theta} +\int_0^\pi \sin\theta|f''(\theta)|^2\,{\rm
d}\theta.  \label{eq:04:theo:eqWm1Y1}
\end{eqnarray}
 
Proceeding similarly, we  derive
\begin{eqnarray}
 I_3&=& -\int_0^\pi \Big(\frac{1}{\sin^2\theta} f(\theta)\Big)'
f'(\theta)\sin\theta\,{\rm d}\theta\nonumber\\
&=&-\int_0^\pi |f'(\theta)|^2\frac{\rm d\theta}{\sin\theta}+
\int_0^\pi \big(f^2(\theta)\big)'\frac{\cos\theta}{\sin^2\theta}\,{\rm
d}\theta\nonumber\\
&=&-\int_0^\pi |f'(\theta)|^2\frac{\rm d\theta}{\sin\theta}+
\int_0^\pi f^2(\theta)
\big(\frac{2\cos^2\theta}{\sin^3\theta}+\frac{1}{\sin\theta}\big)\,{\rm
d}\theta\nonumber\\
&=&-\int_0^\pi |f'(\theta)|^2\frac{\rm d\theta}{\sin\theta}- \int_0^\pi
f^2(\theta) \frac{{\rm
d}\theta}{\sin\theta} +
 2\int_0^\pi f^2(\theta) \frac{{\rm
d}\theta}{\sin^3\theta}. \label{eq:05:theo:eqWm1Y1}
\end{eqnarray}
Inserting \eqref{eq:04:theo:eqWm1Y1}-\eqref{eq:05:theo:eqWm1Y1} in
\eqref{eq:03:theo:eqWm1Y1}, we obtain~\eqref{eq:02:theo:eqWm1Y1}.

The inequalities
\[
\frac{1}{16}+\frac{9m^2}{4}+(m^4-4m^2)\le m^4< 3m^4,\quad \frac14+(1+2m^2)\le
3
m^2,\qquad \forall |m|\ge 2
\]
with \eqref{eq:02:theo:eqWm1Y1}  imply the first
inequality of \eqref{eq:02:cor:equivNorms}.
For $|m|\ge 3$ the second inequality of \eqref{eq:02:cor:equivNorms} is simply
a consequence of the inequalities 
\[
m^4-4m^2> m^4/2> \frac{m^4}{6},\quad 1+2m^2> m^2\ge
\frac{m^2}{6}.
\]
The case
$|m|=2$, has to be analyzed 
separately since one of the crucial terms, the third term in
\eqref{eq:02:theo:eqWm1Y1},   vanishes: Using \eqref{eq:02:theo:eqWm1Y1},
we obtain 
\begin{equation}\label{eq:03.5:cor:equivNormns}
\|f\|_{W_{m}^2}^2\ge 9\int_0^\pi |f(\theta)|^2\,\frac{\rm
d\theta}{\sin\theta}+9\int_0^\pi |f'(\theta)|^2\,\frac{\rm d\theta}{\sin\theta}+
\int_0^\pi |f''(\theta)|^2{\sin\theta}\, {\rm d\theta}. 
\end{equation}
 As before it suffices  to consider $f$ to be real valued and
that  $f\in {\rm span}\left\{ Q_n^2 : n \geq 2 \right\}$.
Note that
\begin{equation}\label{eq:04:cor:equivNormns}
 \int_0^\pi f^2(\theta)\frac{\rm d\theta}{\sin^3\theta}=
 \int_0^\pi f^2(\theta)\frac{\rm d\theta}{\sin\theta}+ \int_0^\pi
f^2(\theta)\frac{\cos^2\theta}{\sin^3\theta}{\rm d\theta}.
\end{equation}
Applying integration by parts to the second term and 
using~\eqref{eq:f:cancellation} 
we obtain
\begin{eqnarray} \label{eq:int_parts}
 \int_0^\pi
f^2(\theta)\frac{\cos^2\theta}{\sin^3\theta}{\rm d\theta}&=&
  \int_0^\pi \big(f(\theta)f'(\theta)\big)
\Big(\log(\tan(\theta/2))+\frac{\cos \theta}{\sin^2\theta}\Big)\,{\rm
d}\theta.
\end{eqnarray}
Notice that for $\theta\in(0,\pi/2]$,
\begin{equation}\label{eq:bound:log}
\sin\theta |\log \tan(\theta/2)|\le 2\tan(\theta/2) |\log \tan(\theta/2) |
 \le 2 e^{-1}\le 1.
\end{equation}
By symmetry, we can extend this bound for any $\theta\in(0,\pi)$.
With the help of \eqref{eq:bound:log} and the inequality  $2ab\le a^2+b^2$, 
from~\eqref{eq:int_parts}  we obtain
\begin{eqnarray}
 \int_0^\pi
f^2(\theta)\frac{\cos^2\theta}{\sin^3\theta}{\rm d\theta}&\le &
  \int_0^\pi   \left| f(\theta)f'(\theta)\right|  \,\frac{{\rm
d}\theta}{\sin\theta}+   \int_0^\pi   \left| f(\theta)f'(\theta)\right| \frac{{\rm
d}\theta}{\sin^2\theta}\nonumber\\
&\le& \frac12\bigg[\int_0^\pi   f^2 (\theta)  \,\frac{{\rm
d}\theta}{\sin\theta} +2\int_0^\pi  |f'(\theta)|^2  \frac{{\rm
d}\theta}{\sin\theta}\bigg] +     \frac{1}2\int_0^\pi   f^2(\theta) \frac{{\rm
d}\theta}{\sin^3\theta}. \qquad\label{eq:05:cor:equivNormns}
\end{eqnarray}
Inserting \eqref{eq:05:cor:equivNormns} in \eqref{eq:04:cor:equivNormns} we
easily  derive
\[
 \int_0^\pi
f^2(\theta)\frac{\rm d\theta}{\sin^3\theta}\le \frac{3}2\int_0^\pi  
f^2(\theta) \,\frac{{\rm
d}\theta}{\sin\theta}+\int_0^\pi  
|f'(\theta)|^2 \,\frac{{\rm
d}\theta}{\sin\theta}+\frac{1}2\int_0^\pi 
f^2(\theta)\frac{\rm d\theta}{\sin^3\theta}
\]
and therefore
\begin{eqnarray}
 \int_0^\pi f^2(\theta)\frac{\rm d\theta}{\sin^3\theta}&\le& 3
\int_0^\pi
f^2(\theta)\frac{\rm d\theta}{\sin\theta}+2\int_0^\pi
|f'(\theta)|^2\frac{{\rm
d}\theta}{\sin\theta}.\label{eq:06:cor:equivNormns}
\end{eqnarray}
From \eqref{eq:03.5:cor:equivNormns} and
\eqref{eq:06:cor:equivNormns}, we   obtain 
\begin{eqnarray*}
 6\|f\|^2_{W_{m}^2}&\ge& 54\int_0^\pi f^2(\theta)\,\frac{\rm
d\theta}{\sin\theta}+54\int_0^\pi |f'(\theta)|^2\,\frac{\rm
d\theta}{\sin\theta}+6
\int_0^\pi |f''(\theta)|^2{\sin\theta}\, {\rm d\theta}\\
&\ge&16\bigg( 3 \int_0^\pi
f^2(\theta)\frac{\rm d\theta}{\sin\theta}+2\int_0^\pi
|f'(\theta)|^2\frac{{\rm
d}\theta}{\sin\theta}\bigg)\\
&&+4\int_0^\pi
|f'(\theta)|^2\frac{{\rm
d}\theta}{\sin\theta}+
\int_0^\pi |f''(\theta)|^2{\sin\theta}\, {\rm d\theta}\\
&\ge&16\int_0^\pi
|f (\theta)|^2\frac{{\rm
d}\theta}{\sin^3\theta}+4\int_0^\pi
|f'(\theta)|^2\frac{{\rm
d}\theta}{\sin\theta}+
\int_0^\pi |f''(\theta)|^2{\sin\theta}\, {\rm d\theta}.
\end{eqnarray*}

Hence  the inequalities in~\eqref{eq:02:cor:equivNorms} hold.
\end{proof}

\subsection{Proof of Proposition~\ref{prop:inclusion_in_Hr}}
\label{sec:proposition3.2}
\begin{proof} 
Denote by $\Gamma$ the maximum circle in $\Sp$, parametrized by
 \begin{equation}
 \label{eq:chi}
\q(\theta):=(\sin\theta, 0 ,\cos\theta).
 \end{equation}
Given $f^\circ:\Gamma\to \mathbb{C}$   we denote   $f=f^\circ\circ 
\q:\R\to\mathbb{C}$.   The norm  in the Sobolev space  $H^r(\Gamma)$  can 
be then
defined with the help of $\q$ and~\eqref{eq:SobolevNorm}: 
\[
\|f^\circ\|_{H^r(\Gamma)}:=\|f\|_{H_{\#}^r}.
\]

The second ingredient we will use in this proof is the trace
operator $\gamma_\Gamma$ which can be shown to be  continuous from ${\cal
H}^{r+1/2}(\Sp)$ onto $H^{r}(\Gamma)$ for all
$r>0$ (see \cite{MR937473, McLean:2000} for a proof of this result in $\R^n$;
the proof can be easily extended by using local charts of the unit sphere and
the equivalent definitions of the Sobolev spaces involved).

Given $f\in W_m^r$, consider the mapping
\[
 {\cal P}_m f:=\sqrt{2\pi}(\gamma_\Gamma F^\circ )\circ \q, \qquad
F^\circ:=(f\otimes e_m) \circ \p^{-1}.
\]
Observe that $F^\circ\in {\cal H}^r(\Sp)$ and that actually $f= {\cal P}_m f$,
that is ${\cal P}_m$ is simply the identity operator. 
Moreover,
\[
\| {\cal P}_m f\|_{H_\#^r} \le
{\sqrt{2\pi}}\|\gamma_\Gamma\|_{{\cal H}^{r+1/2}(\Sp)\to
H^r(\Gamma)}\|F^\circ\|_{{\cal
H}^{r+1/2}(\Sp)}={\sqrt{2\pi}}\|\gamma_\Gamma\|_{{\cal H}^{r+1/2}(\Sp)\to
H^r(\Gamma)} \|f\|_{W_m^{r+1/2}},
\]
where $\|\gamma_\Gamma\|_{{\cal H}^{r+1/2}(\Sp)\to
H^r(\Gamma)}$ is the continuity constant of $\gamma_\Gamma$ as a linear  
operator from
${\cal  H}^{r+1/2}(\Sp)$ onto ${H^{r}(\Gamma)}$. 
Hence we obtain \eqref{eq:01:prop:inclusion_in_Hr}.

Since $W_m^0\cong
L^2_{\sin}$,  the first equation in \eqref{eq:03a:prop:inclusion_in_Hr} is
clear whereas  the second equation in \eqref{eq:03b:prop:inclusion_in_Hr}  for $m=0$ follows directly
from
 \eqref{eq:01:theo:eqWm1Y1} and \eqref{eq:SobolevNorm:2}.
Finally, if $m\ne 0$   using $f(0)=0$ 
we observe that
\begin{eqnarray*}
 \int_0^{\pi/2}|f(\theta)|^2\frac{\rm d\theta}{\rm \sin\theta}&=&
 \int_0^{\pi/2}\frac1{\sin \theta}\bigg|\int_0^\theta f'(\xi)\,{\rm d}\xi
\bigg|^2 {\rm d \theta}\le \int_0^{\pi/2} \frac{\sqrt{\theta}}{\sin
\theta}\bigg[
\int_0^\theta |f'(\xi)|^2\,{\rm d}\xi\bigg]\,{\rm d}\theta\\
&\le& C \int_0^{\pi/2} |f'(\xi)|^2\,{\rm
d}\xi.
\end{eqnarray*} 
Proceeding similarly, but using now that
$f(\pi)=0$,
we can bound   the integral in $(\pi/2,\pi)$ and  hence  conclude that 
\[
  \int_0^{\pi}|f(\theta)|^2\frac{\rm d\theta}{\rm \sin\theta}\le C
\int_0^{\pi} |f'(\xi)|^2\,{\rm
d}\xi. 
\]
Equation \eqref{eq:01:cor:equivNorms}  now yields that
\[
 \|f\|_{W_m^1}\le\frac{\sqrt{5}}2\|f\|_{{Z_m^1}}\le C(1+|m|)\|f\|_{H_\#^1}. 
\]

\end{proof}

\section*{Acknowledgment}

The research of the first author was supported, in part, by grant DMS-1216889
from the National Science Foundation and by Ministerio de Econom\'{\i}a y Competitividad through the grant MTM2014-52859.  Support of the Colorado Golden Energy
Computing Organization (GECO) is gratefully acknowledged.

\bibliography{biblio}

\end{document}